\documentclass[a4paper]{amsart}

\usepackage[utf8]{inputenc}
\usepackage[english]{babel}

\usepackage[T1]{fontenc}
\usepackage{lmodern}  

\usepackage{hyperref}
\usepackage{my-math-symbol}
\usepackage{my-cleveref}
\usepackage{my-equation-numbering}
\usepackage{comment}
\usepackage{tikz}
\usetikzlibrary{cd}
\usepackage[colorinlistoftodos]{todonotes}

\newcommand{\wtcS}{\widetilde{\calS}}

\usepackage[non-sorted-cites,initials,alphabetic,nobysame]{amsrefs}

\AtBeginDocument{%
	\def\MR#1{}
}

\title[Generalizations of the $Q$-prime curvature]{Generalizations of the $Q$-prime curvature via renormalized characteristic forms}
\author{Yuya Takeuchi}
\address{Division of Mathematics \\ Faculty of Pure and Applied Sciences \\ University of Tsukuba
	\\ 1-1-1 Tennodai, Tsukuba, Ibaraki 305-8571 Japan}
\email{ytakeuchi@math.tsukuba.ac.jp, yuya.takeuchi.math@gmail.com}



\thanks{This work was supported by JSPS KAKENHI Grant Number JP21K13792.}

\begin{document}

\begin{abstract}
	The $Q$-prime curvature is a local pseudo-Einstein invariant on CR manifolds
	defined by Case and Yang, and Hirachi.
	Its integral,
	the total $Q$-prime curvature,
	gives a non-trivial global CR invariant.
	On the other hand,
	Marugame has constructed a family of global CR invariants
	via renormalized characteristic forms,
	which contains the total $Q$-prime curvature.
	In this paper,
	we introduce a generalization of the $Q$-prime curvature
	for each renormalized characteristic form,
	and show that its integral coincides with Marugame's CR invariant.
	We also study generalizations of the critical CR GJMS operator and the $P$-prime operator,
	which are related to the transformation laws of our new curvatures under conformal change.
\end{abstract}

\maketitle

\section{Introduction}
\label{section:introduction}

In the seminal work~\cite{Fefferman1974},
Fefferman has proved that
two bounded strictly pseudoconvex domains in $\bbC^{n + 1}$ are biholomorphic
if and only if their boundaries,
which are strictly pseudoconvex CR manifolds,
are CR equivalent.
Since then,
there have been extensive researches on invariants for strictly pseudoconvex CR manifolds.
Here we give some examples of global CR invariants that are related to our results.
For simplicity,
we consider only the boundary $M$ of a bounded strictly pseudoconvex domain $\Omega \subset \bbC^{n + 1}$
in this section.

The first example is the boundary term of the renormalized Gauss-Bonnet-Chern formula.
Fefferman~\cite{Fefferman1976} has constructed a defining function $\rho$ of $\Omega$
solving an asymptotic complex Monge-\Ampere equation.
Consider the \Kahler form
\begin{equation}
	\omega_{+}
	\coloneqq - d d^{c} \log (- \rho)
\end{equation}
near the boundary,
where $d^{c} = (\sqrt{- 1} / 2) (\delb - \del)$.
The Chern connection with respect to $\omega_{+}$ diverges on the boundary
since so does $\omega_{+}$.
Burns and Epstein~\cite{Burns-Epstein1990-Char} have introduced
a renormalization procedure for this connection,
which gives a connection smooth up to the boundary.
Note that this procedure is an example of $c$-projective compactifications~\cite{Cap-Gover2019}.
Denote by $\Theta$ the curvature form with respect to the renormalized connection.
Burns and Epstein has proved the following renormalized Gauss-Bonnet-Chern formula
by the method of homological sections:
\begin{equation}
	\int_{\Omega} c_{n + 1}(\Theta)
	= \chi(\Omega) + \mu(M).
\end{equation}
Here $\mu(M)$ is the boundary correction,
which gives a global CR invariant of $M$.
Marugame~\cite{Marugame2016} has generalized this $\mu(M)$
to strictly pseudoconvex CR manifolds admitting pseudo-Einstein contact forms
by a similar argument to Chern's original proof of the Gauss-Bonnet-Chern formula.

The next example is the total $Q$-prime curvature.
We first refer to the critical CR GJMS operator and the $P$-prime operator for the later use.
Gover and Graham~\cite{Gover-Graham2005} have introduced the \emph{critical CR GJMS operator} $P$
by using the Fefferman conformal structure.
This operator is a formally self-adjoint CR invariant linear differential operator.
It follows from the definition that $P$ annihilates CR pluriharmonic functions.
Based on this fact,
Case and Yang~\cite{Case-Yang2013} and Hirachi~\cite{Hirachi2014}
have defined the \emph{$P$-prime operator} $P^{\prime}$,
the ``secondary'' version of $P$.
This is a linear differential operator acting on CR pluriharmonic functions,
and transforms as follows under the conformal change $\whxth = e^{\Upsilon} \theta$:
\begin{equation}
\label{eq:transformation-of-P-prime}
	e^{(n + 1) \Upsilon} \whP^{\prime} f
	= P^{\prime} f + P(\Upsilon f),
\end{equation}
where $\whP^{\prime}$ is defined in terms of $\whxth$.
Moreover,
they have introduced the \emph{$Q$-prime curvature} $Q^{\prime}$,
the ``secondary'' version of the CR $Q$-curvature~\cite{Fefferman-Hirachi2003}.
The $Q$-prime curvature is a smooth function defined for each pseudo-Einstein contact form,
and has the following transformation rule
under the change of pseudo-Einstein contact forms $\whxth = e^{\Upsilon} \theta$:
\begin{equation}
\label{eq:transformation-of-Q-prime}
	e^{(n + 1) \Upsilon} \whQ^{\prime}
	= Q^{\prime} + 2 P^{\prime} \Upsilon + P \Upsilon^{2},
\end{equation}
where $\whQ^{\prime}$ is defined in terms of $\whxth$.
Note that $\Upsilon$ is a CR pluriharmonic function.
Moreover,
the integral $\ovQ^{\prime} \coloneqq \int_{M} Q^{\prime} \, \theta \wedge (d \theta)^{n}$,
the \emph{total $Q$-prime curvature},
is independent of the choice of pseudo-Einstein contact forms
and gives a global CR invariant of $M$~\cites{Case-Yang2013,Hirachi2014,Marugame2018}.
Furthermore,
this invariant satisfies the equality
\begin{equation}
	\lp \int_{\rho < - \varepsilon} d \log (- \rho) \wedge d^{c} \log (- \rho) \wedge \omega_{+}^{n}
	= \frac{(- 1)^{n}}{2 (n!)^{2}} \ovQ^{\prime},
\end{equation}
where $\lp$ stands for the coefficient of the $\log \varepsilon$ term.

The last example is a family of ``renormalized characteristic numbers.''
Let $\Phi$ be an $\Ad$-invariant polynomial on $\mathfrak{gl}(n + 1, \bbC)$
homogeneous of degree $m$ with $0 \leq m \leq n$.
Marugame~\cite{Marugame2021} has proved that
\begin{equation}
	\scrI_{\Phi}
	\coloneqq \lp \int_{\rho < - \varepsilon} d \log (- \rho) \wedge d^{c} \log (- \rho)
		\wedge \omega_{+}^{n - m} \wedge \Phi(\Theta)
\end{equation}
defines a global CR invariant of $M$.
Note that the total $Q$-prime curvature corresponds to the case of $m = 0$.
Moreover if $m = n$,
he has shown that $\scrI_{\Phi}$ coincides with the integral of the $\calI_{\Phi}$-prime curvature,
which has been introduced by Case and Gover~\cite{Case-Gover2020} for $n = 2$,
and generalized to all dimensions by Marugame~\cite{Marugame2021}
and Case and the author~\cite{Case-Takeuchi2023} independently.

The aim of this paper is to introduce a generalization of the $Q$-prime curvature for each $\Phi$
such that its integral gives the invariant $\scrI_{\Phi}$.
To this end,
we consider the Dirichlet problem of the $\delb$-Laplacian with respect to $\omega_{+}$.

We will first generalize the critical CR GJMS operator $P$.
Unlike the case of $\deg \Phi = 0$,
there exist two possible generalizations.
One is the \emph{$P_{\Phi}$-operator} $P_{\Phi}$ (\cref{def:P_Phi-operator}),
which is a symmetric linear differential operator.
The other is the \emph{$\calP_{\Phi}$-operator} $\calP_{\Phi}$ (\cref{def:calP_Phi-operator}),
which is a symmetric bilinear differential operator.

Similar to the critical CR GJMS operator,
$P_{\Phi}$ annihilates CR pluriharmonic functions.
Thus we can define the \emph{$P_{\Phi}$-prime operator} $P_{\Phi}^{\prime}$,
the ``secondary'' version of the $P_{\Phi}$-operator (\cref{def:P_Phi-prime-operator}).
This $P_{\Phi}^{\prime}$ is a linear differential operator acting on CR pluriharmonic functions.
Moreover,
it satisfies a similar transformation law to \cref{eq:transformation-of-P-prime}
under the conformal change $\whxth = e^{\Upsilon} \theta$:
\begin{equation}
	e^{(n + 1) \Upsilon} \whP_{\Phi}^{\prime} f
	= P_{\Phi}^{\prime} f + \calP_{\Phi}(\Upsilon, f),
\end{equation}
where $\whP_{\Phi}^{\prime}$ is defined in terms of $\whxth$ (\cref{prop:transformation-law-of-P-prime-operator}).

As noted above,
we will also introduce the \emph{$Q_{\Phi}$-prime curvature} $Q_{\Phi}^{\prime}$
(\cref{def:Q_Phi-prime-curvature}).
This is a smooth function defined for each pseudo-Einstein contact form,
and has an analogous transformation rule to \cref{eq:transformation-of-Q-prime}
under the change of pseudo-Einstein contact forms $\whxth = e^{\Upsilon} \theta$:
\begin{equation}
	e^{(n + 1) \Upsilon} \whQ_{\Phi}^{\prime}
	= Q_{\Phi}^{\prime} + 2 P_{\Phi}^{\prime} \Upsilon + \calP_{\Phi}(\Upsilon, \Upsilon),
\end{equation}
where $\whQ_{\Phi}^{\prime}$ is defined in terms of $\whxth$
(\cref{prop:transformation-law-of-Q-prime-curvature}).
Moreover,
the integral of $Q_{\Phi}^{\prime}$,
the \emph{total $Q_{\Phi}$-prime curvature},
has similar properties to the total $Q$-prime curvature.

\begin{theorem}
\label{thm:total-Q_Phi-prime-curvature}
	The integral
	\begin{equation}
		\ovQ_{\Phi}^{\prime}
		\coloneqq \int_{M} Q_{\Phi}^{\prime} \, \theta \wedge (d \theta)^{n}
	\end{equation}
	is independent of the choice of a pseudo-Einstein contact form $\theta$,
	and defines a global CR invariant of $M$.
	Moreover,
	this invariant satisfies the following equality:
	\begin{equation}
		\scrI_{\Phi}
		= (- 1)^{n} \frac{(n + 1)^{2}}{2} \ovQ_{\Phi}^{\prime}.
	\end{equation}
\end{theorem}

It is worth noting that our proof of the CR invariance of $\ovQ_{\Phi}^{\prime}$
is based on the transformation law of $Q_{\Phi}^{\prime}$,
which is independent of Marugame's one.
In addition,
we will discuss a relation between the $Q_{\Phi}$-prime curvature
and the $\calI_{\Phi}$-prime curvature in the case of $\deg \Phi = n$,
and compute $Q_{\Phi}^{\prime}$ on Sasakian $\eta$-Einstein manifolds when $\deg \Phi = n - 1$.

This paper is organized as follows.
In \cref{section:pseudo-Hermitian-geometry} (resp.\ \cref{section:strictly-pseudoconvex-domains}),
we recall basic facts on CR manifolds (resp.\ strictly pseudoconvex domains).
\cref{section:P_Phi-operator-and-calP_Phi-operator} provides the definitions of $P_{\Phi}$ and $\calP_{\Phi}$.
In \cref{section:P_Phi-prime-operator},
we introduce the $P_{\Phi}$-prime operator.
\cref{section:Q_Phi-prime-curvature} is devoted to the definition of the $Q_{\Phi}$-prime curvature
and the proof of \cref{thm:total-Q_Phi-prime-curvature}.
In \cref{section:deg-Phi=n-case},
we compute explicit formulae of $P_{\Phi}^{\prime}$ and $Q_{\Phi}^{\prime}$ for $\deg \Phi = n$
and compare these with $X^{\Phi}_{\alpha}$ and $\calI^{\prime}_{\Phi}$
introduced in \cites{Marugame2021,Case-Takeuchi2023}.
\cref{section:deg-Phi=n-1-case-on-SE} deals with the case of $\deg \Phi = n - 1$
on Sasakian $\eta$-Einstein manifolds.

\medskip

\noindent
\emph{Notation.}
We use Einstein's summation convention and assume that
\begin{itemize}
	\item lowercase Greek indices $\alpha, \beta, \gamma, \dots$ run from $1, \dots, n$;
	\item lowercase Latin indices $a, b, c, \dots$ run from $1, \dots , n, \infty$.
\end{itemize}

Suppose that a function $I(\varepsilon)$ admits an asymptotic expansion,
as $\varepsilon \to + 0$,
\begin{equation}
	I(\varepsilon)
	=\sum_{m = 1}^{k} a_{m} \varepsilon^{-m} + b \log \varepsilon + O(1).
\end{equation}
Then the \emph{logarithmic part} $\lp I(\varepsilon)$ of $I(\varepsilon)$ is the constant $b$.

\medskip

\section{Pseudo-Hermitian geometry}
\label{section:pseudo-Hermitian-geometry}

\subsection{CR structures}
\label{subsection:CR-structures}

Let $M$ be a smooth $(2n+1)$-dimensional manifold without boundary.
A \emph{CR structure} is a rank $n$ complex subbundle $T^{1, 0} M$
of the complexified tangent bundle $T M \otimes \mathbb{C}$ such that
\begin{equation}
	T^{1, 0} M \cap T^{0, 1} M = 0, \qquad
	[\Gamma(T^{1, 0} M), \Gamma(T^{1, 0} M)] \subset \Gamma(T^{1, 0} M),
\end{equation}
where $T^{0, 1} M$ is the complex conjugate of $T^{1, 0} M$ in $T M \otimes \mathbb{C}$.
Set $H M = \Re T^{1, 0}M$
and let $J \colon H M \to H M$ be the unique complex structure on $H M$ 
such that
\begin{equation}
	T^{1, 0} M = \Ker(J - \sqrt{- 1} \colon H M \otimes \mathbb{C} \to H M \otimes \mathbb{C}).
\end{equation}
A typical example of CR manifolds is a real hypersurface $M$ in an $(n + 1)$-dimensional complex manifold $X$;
this $M$ has the induced CR structure
\begin{equation}
	T^{1, 0} M
	\coloneqq T^{1, 0} X |_{M} \cap (T M \otimes \mathbb{C}).
\end{equation}
In particular,
the unit sphere
\begin{equation}
	S^{2 n + 1}
	\coloneqq \Set{z \in \bbC^{n + 1} | \abs{z}^{2} = 1}
\end{equation}
has the canonical CR structure $T^{1, 0} S^{2 n + 1}$.

Introduce an operator $\delb_{b} \colon C^{\infty}(M) \to \Gamma((T^{0, 1} M)^{*})$ by
\begin{equation}
	\delb_{b} f
	\coloneqq (d f)|_{T^{0, 1} M}.
\end{equation}
A smooth function $f$ is called a \emph{CR holomorphic function}
if $\delb_{b} f = 0$.
A \emph{CR pluriharmonic function} is a real-valued smooth function
that is locally the real part of a CR holomorphic function.
We denote by $\scrP$ the space of CR pluriharmonic functions.

A CR structure $T^{1, 0} M$ is said to be \emph{strictly pseudoconvex}
if there exists a nowhere-vanishing real one-form $\theta$ on $M$
such that
$\theta$ annihilates $T^{1, 0} M$ and
\begin{equation}
	- \sqrt{-1} d \theta (Z, \overline{Z}) > 0, \qquad
	0 \neq Z \in T^{1, 0} M.
\end{equation}
We call such a one-form a \emph{contact form}.
The triple $(M, T^{1, 0} M, \theta)$ is called a \emph{pseudo-Hermitian manifold}.
Denote by $T$ the \emph{Reeb vector field} with respect to $\theta$; 
that is,
the unique vector field satisfying
\begin{equation}
	\theta(T) = 1, \qquad T \contr d\theta = 0.
\end{equation}
Let $(Z_{\alpha})$ be a local frame of $T^{1, 0} M$,
and set $Z_{\ovxa} = \overline{Z_{\alpha}}$.
Then
$(T, Z_{\alpha}, Z_{\ovxa})$ gives a local frame of $T M \otimes \mathbb{C}$,
called an \emph{admissible frame}.
Its dual frame $(\theta, \theta^{\alpha}, \theta^{\ovxa})$
is called an \emph{admissible coframe}.
The two-form $d \theta$ is written as
\begin{equation}
	d \theta
	= \sqrt{-1} l_{\alpha \ovxb} \theta^{\alpha} \wedge \theta^{\ovxb},
\end{equation}
where $(l_{\alpha \ovxb})$ is a positive definite Hermitian matrix.
We use $l_{\alpha \ovxb}$ and its inverse $l^{\alpha \ovxb}$
to raise and lower indices of tensors.

\subsection{Tanaka-Webster connection and pseudo-Einstein condition}
\label{subsection:TW-connection-and-pE-condition}

A contact form $\theta$ induces a canonical connection $\nabla$ on $T M$,
called the \emph{Tanaka-Webster connection} with respect to $\theta$.
It is defined by
\begin{equation}
	\nabla T
	= 0,
	\quad
	\nabla Z_{\alpha}
	= \omega_{\alpha} {}^{\beta} Z_{\beta},
	\quad
	\nabla Z_{\ovxa}
	= \omega_{\ovxa} {}^{\ovxb} Z_{\ovxb}
	\quad
	\rbra*{\omega_{\ovxa} {}^{\ovxb}
	= \overline{\omega_{\alpha} {}^{\beta}}}
\end{equation}
with the following structure equations:
\begin{gather}
\label{eq:str-eq-of-TW-conn1}
	d \theta^{\beta}
	= \theta^{\alpha} \wedge \omega_{\alpha} {}^{\beta}
	+ A^{\beta} {}_{\ovxa} \theta \wedge \theta^{\ovxa}, \\
\label{eq:str-eq-of-TW-conn2}
	d l_{\alpha \ovxb}
	= \omega_{\alpha} {}^{\gamma} l_{\gamma \ovxb}
		+ l_{\alpha \ovxg} \omega_{\ovxb} {}^{\ovxg}.
\end{gather}
The tensor $A_{\alpha \beta} = \overline{A_{\ovxa \ovxb}}$
is shown to be symmetric and is called the \emph{Tanaka-Webster torsion}.
We denote the components of a successive covariant derivative of a tensor
by subscripts preceded by a comma,
for example, $K_{\alpha \ovxb , \gamma}$;
we omit the comma if the derivatives are applied to a function.
The \emph{sub-Laplacian} $\Delta_{b}$ is defined by
\begin{equation}
	\Delta_{b} u
	\coloneqq - u_{\alpha} {}^{\alpha} - u_{\ovxb} {}^{\ovxb}.
\end{equation}

The curvature form
$\Omega_{\alpha} {}^{\beta} \coloneqq d \omega_{\alpha} {}^{\beta}
- \omega_{\alpha} {}^{\gamma} \wedge \omega_{\gamma} {}^{\beta}$
of the Tanaka-Webster connection satisfies
\begin{equation}
\label{eq:curvature-form-of-TW-connection}
	\Omega_{\alpha} {}^{\beta}
	= R_{\alpha} {}^{\beta} {}_{\rho \ovxs} \theta^{\rho} \wedge \theta^{\ovxs}
		\qquad \text{modulo $\theta, \theta^{\rho} \wedge \theta^{\gamma},
		\theta^{\ovxg} \wedge \theta^{\ovxs}$}.
\end{equation}
We call the tensor $R_{\alpha} {}^{\beta} {}_{\rho \ovxs}$
the \emph{Tanaka-Webster curvature}.
This tensor has the symmetry 
\begin{equation}
	R_{\alpha \ovxb \rho \ovxs}
	= R_{\rho \ovxb \alpha \ovxs}
	= R_{\alpha \ovxs \rho \ovxb}.
\end{equation}
Contraction of indices gives the \emph{Tanaka-Webster Ricci curvature}
$\Ric_{\rho \ovxs} = R_{\alpha} {}^{\alpha} {}_{\rho \ovxs}$
and the \emph{Tanaka-Webster scalar curvature}
$\Scal = \Ric_{\rho} {}^{\rho}$.
When $n > 1$,
the \emph{Chern tensor} $S_{\alpha \ovxb \rho \ovxs}$ is the completely trace-free part
of $R_{\alpha \ovxb \rho \ovxs}$;
this tensor is a CR analogue of the Weyl tensor in conformal geometry.
It is known that the Chern tensor vanishes identically
if and only if $(M, T^{1, 0} M)$ is spherical;
that is,
locally CR equivalent to $(S^{2 n + 1}, T^{1, 0} S^{2 n + 1})$~\cite{Chern-Moser1974}.

A contact form $\theta$ is said to be \emph{pseudo-Einstein}
if the following two equalities hold:
\begin{equation}
\label{eq:pseudo-Einstein-condition}
	\Ric_{\alpha \ovxb}
	= \frac{1}{n} \Scal \cdot l_{\alpha \ovxb},
	\qquad
	\Scal_{\alpha}
	= \sqrt{- 1} n A_{\alpha \beta ,} {}^{\beta}.
\end{equation}
From Bianchi identities for the Tanaka-Webster connection,
we obtain
\begin{equation}
	\rbra*{\Ric_{\alpha \ovxb} - \frac{1}{n} \Scal \cdot l_{\alpha \ovxb}}_{,} {}^{\ovxb}
	= \frac{n - 1}{n} \rbra*{\Scal_{\alpha} - \sqrt{- 1} n A_{\alpha \beta ,} {}^{\beta}};
\end{equation}
see~\cite{Hirachi2014}*{Lemma 5.7(iii)} for example.
Hence the latter equality of \cref{eq:pseudo-Einstein-condition} follows from the former one if $n > 1$.
On the other hand,
the former equality of \cref{eq:pseudo-Einstein-condition} automatically holds if $n = 1$,
and the latter one is a non-trivial condition.
It is known that $\whxth = e^{\Upsilon} \theta$ is another pseudo-Einstein contact form
if and only if $\Upsilon$ is CR pluriharmonic;
see~\cite{Lee1988}*{Proposition 5.1} for example.

\subsection{Sasakian manifolds}
\label{subsection:Sasakian-manifolds}

Sasakian manifolds constitute an important class of pseudo-Hermitian manifolds.
See~\cite{Boyer-Galicki2008} for a comprehensive introduction to Sasakian manifolds.

A \emph{Sasakian manifold} is a pseudo-Hermitian manifold $(S, T^{1, 0}S, \eta)$
with vanishing Tanaka-Webster torsion.
Remark that we use $\eta$ instead of $\theta$
since this is the common notation in Sasakian geometry.
This condition is equivalent to that the Reeb vector field $T$ with respect to $\eta$
preserves the CR structure $T^{1, 0} S$.
An almost complex structure $I$ on the cone $C(S) = \bbR_{+} \times S$ of $S$
is defined by
\begin{equation}
	I(a (r \pdvf{}{r}) + b T + V) = - b (r \pdvf{}{r}) + a T + J V,
\end{equation}
where $r$ is the coordinate of $\bbR_{+}$,
$a, b \in \mathbb{R}$, and $V \in H S$.
The bundle $T^{1, 0} C(S)$ of $(1, 0)$-vectors with respect to $I$ is given by
\begin{equation}
	T^{1, 0} C(S)
	= \mathbb{C}(r \pdvf{}{r} - \sqrt{- 1} T) \oplus T^{1, 0} S.
\end{equation}
The vanishing of the Tanaka-Webster torsion implies that
$I$ is integrable;
that is,
$(C(S), I)$ is a complex manifold.
Moreover,
the one-form $\eta$ is equal to $d^{c} \log r^{2}$.
In what follows,
we identify $S$ with the level set $\{1\} \times S \subset C(S)$.

A $(2 n + 1)$-dimensional Sasakian manifold $(S, T^{1, 0} S, \eta)$
is said to be \emph{Sasakian $\eta$-Einstein with Einstein constant $(n + 1) \lambda$}
if the Tanaka-Webster Ricci curvature satisfies
\begin{equation}
	\Ric_{\alpha \ovxb}
	= (n + 1) \lambda l_{\alpha \ovxb}.
\end{equation}
Note that $\eta$ is a pseudo-Einstein contact form on $(S, T^{1, 0} S)$.

\section{Asymptotic analysis on strictly pseudoconvex domains}
\label{section:strictly-pseudoconvex-domains}

Let $\Omega$ be a relatively compact domain in an $(n+1)$-dimensional complex manifold $X$
with smooth boundary $M = \bdry \Omega$.
There exists a smooth function $\rho$ on $X$ such that
\begin{equation}
	\Omega = \rho^{-1}((- \infty, 0)), \quad
	M = \rho^{-1}(0), \quad
	d \rho |_{M} \neq 0;
\end{equation}
such a $\rho$ is called a \emph{defining function} of $\Omega$.
A domain $\Omega$ is said to be \emph{strictly pseudoconvex}
if we can take a defining function $\rho$ of $\Omega$ that is strictly plurisubharmonic near $M$.
The boundary $M$ is a closed strictly pseudoconvex real hypersurface
and $d^{c} \rho |_{T M}$ is a contact form on $M$.
Conversely,
it is known that
any closed connected strictly pseudoconvex CR manifold of dimension at least five can be realized as
the boundary of a strictly pseudoconvex domain
in a complex projective manifold~\cites{Boutet_de_Monvel1975,Harvey-Lawson1975,Lempert1995}.

\subsection{Graham-Lee connection}
\label{subsection:Graham-Lee-connection}

Let $\rho$ be a defining function of a strictly pseudoconvex domain $\Omega$
in an $(n + 1)$-dimensional complex manifold $X$.
By strict pseudoconvexity,
there exists the unique $(1, 0)$-vector field $\wtZ_{\infty}$ near the boundary such that
\begin{equation}
	\wtZ_{\infty} \rho = 1,
	\qquad
	\wtZ_{\infty} \contr \del \delb \rho
	= \kappa \delb \rho
\end{equation}
for a smooth function $\kappa$,
which is called the \emph{transverse curvature}.
Set
\begin{equation}
	N
	\coloneqq \Re \wtZ_{\infty},
	\qquad
	\wtT
	\coloneqq - 2 \Im \wtZ_{\infty}.
\end{equation}
Then $N$ and $\wtT$ satisfy
\begin{equation}
	N \rho
	= 1,
	\qquad
	\vartheta(N)
	= 0,
	\qquad
	\wtT \rho
	= 0,
	\qquad
	\vartheta(\wtT)
	= 1,
	\qquad
	(\wtT \contr d \vartheta)|_{\Ker d \rho}
	= 0,
\end{equation}
where $\vartheta \coloneqq d^{c} \rho$.
In particular,
$\wtT$ coincides with the Reeb vector field with respect to $\vartheta$ on each level set of $\rho$.
Take a local frame $(\wtZ_{\alpha})$ of $\Ker \del \rho$,
and let $(\wtxth^{\alpha}, \wtxth^{\infty} = \del \rho)$ be
the dual frame of $(\wtZ_{\alpha}, \wtZ_{\infty})$.
We set $\wtZ_{\ovb} \coloneqq \overline{\wtZ_{b}}$
and $\wtxth^{\ovb} \coloneqq \overline{\wtxth^{b}}$.
Then
\begin{equation}
\label{eq:derivative-of-vartheta}
	d \vartheta
	= \sqrt{- 1} \wtl_{\alpha \ovxb} \wtxth^{\alpha} \wedge \wtxth^{\ovxb}
		+ \kappa d \rho \wedge \vartheta,
\end{equation}
where $\wtl_{\alpha \ovxb} \coloneqq \del \delb \rho (\wtZ_{\alpha}, \wtZ_{\ovxb})$
is the Levi form on each level set of $\rho$ for the contact form
given by the restriction of $\vartheta$.
To simplify notation,
we set
\begin{equation}
	\wtxm
	\coloneqq \sqrt{- 1} \wtl_{\alpha \ovxb} \wtxth^{\alpha} \wedge \wtxth^{\ovxb}.
\end{equation}
We use $\wtl_{\alpha \ovxb}$ and its inverse $\wtl^{\alpha \ovxb}$
to raise and lower indices of tensors.
The \emph{Graham-Lee connection} $\wtna$
is the unique connection on $T X$ defined by
\begin{equation}
	\wtna \wtZ_{\alpha}
	= \wtxo_{\alpha} {}^{\beta} \wtZ_{\beta},
	\quad
	\wtna \wtZ_{\ovxa}
	= \wtxo_{\ovxa} {}^{\ovxb} \wtZ_{\ovxb}
	\quad
	\rbra*{\wtxo_{\ovxa} {}^{\ovxb}
	= \overline{\wtxo_{\alpha} {}^{\beta}}},
	\quad
	\wtna \wtZ_{\infty}
	= \widetilde{\nabla} \wtZ_{\ovin}
	= 0,
\end{equation}
with the following structure equations:
\begin{gather}
\label{eq:derivative-of-admissible-frame}
	d \wtxth^{\beta}
	= \wtxth^{\alpha} \wedge \wtxo_{\alpha} {}^{\beta}
		- \sqrt{- 1} \wtA^{\beta} {}_{\ovxs} \del \rho \wedge \wtxth^{\ovxs}
		+ \sqrt{- 1} \kappa^{\beta} d \rho \wedge \vartheta
		+ \frac{1}{2} \kappa d \rho \wedge \wtxth^{\beta}, \\
	d \wtl_{\alpha \ovxb}
	= \wtxo_{\alpha} {}^{\gamma} \wtl_{\gamma \ovxb}
		+ \wtl_{\alpha \ovxg} \wtxo_{\ovxb} {}^{\ovxg};
\end{gather}
see~\cite{Graham-Lee1988}*{Proposition 1.1} for the proof of the existence and uniqueness.
Note that the restriction of $\wtna$ to each level set of $\rho$
coincides with the Tanaka-Webster connection on it.
We also note that
\begin{equation}
\label{eq:parallel-differential-forms}
	\wtna d \rho
	= 0,
	\qquad
	\wtna \vartheta
	= 0,
	\qquad
	\wtna \wtxm
	= 0.
\end{equation}
The following lemma will be used later.

\begin{lemma}
	For any $u \in C^{\infty}(\ovxco)$,
	\begin{equation}
	\label{eq:dd^c(NT)}
		d d^{c} u (N, \wtT)
		= \rbra*{N^{2} + \frac{1}{4} \wtT^{2} + \kappa N
			+ \frac{1}{2} \kappa^{\alpha} \wtZ_{\alpha}
			+ \frac{1}{2} \kappa^{\ovxb} \wtZ_{\ovxb}} u.
	\end{equation}
\end{lemma}

\begin{proof}
	See the proof of~\cite{Graham-Lee1988}*{Proposition 2.1}.
\end{proof}

The curvature form
$\wtxco_{\alpha} {}^{\beta} \coloneqq
d \wtxo_{\alpha} {}^{\beta} - \wtxo_{\alpha} {}^{\gamma} \wedge \wtxo_{\gamma} {}^{\beta}$
is given by
\begin{equation}
\label{eq:TW-curvature-form}
	\begin{split}
			\wtxco_{\alpha} {}^{\beta}
		&= \wtR_{\alpha} {}^{\beta} {}_{\rho \ovxs} \wtxth^{\rho} \wedge \wtxth^{\ovxs}
			+ \sqrt{- 1} \wtA_{\alpha \gamma ,} {}^{\beta} \wtxth^{\gamma} \wedge \delb \rho
			+ \sqrt{- 1} \wtA^{\beta} {}_{\ovxg , \alpha} \wtxth^{\ovxg} \wedge \del \rho \\
		&\quad - \sqrt{- 1} \wtA_{\alpha \gamma} \wtxth^{\gamma} \wedge \wtxth^{\beta}
			+ \sqrt{- 1} \wtl_{\alpha \ovxg} \wtA^{\beta} {}_{\ovxd} \wtxth^{\ovxg} \wedge \wtxth^{\ovxd} \\
		&\quad + d \rho \wedge \rbra*{\kappa_{\alpha} \wtxth^{\beta}
			- \wtl_{\alpha \ovxg} \kappa^{\beta} \wtxth^{\ovxg}
			+ \frac{1}{2} \delta_{\alpha} {}^{\beta} \kappa_{\gamma} \wtxth^{\gamma}
			- \frac{1}{2} \delta_{\alpha} {}^{\beta} \kappa_{\ovxg} \wtxth^{\ovxg}} \\
		&\quad + \frac{\sqrt{- 1}}{2} (\kappa_{\alpha} {}^{\beta} + \kappa^{\beta} {}_{\alpha}
			+ 2 \wtA_{\alpha \gamma} \wtA^{\gamma \beta}) d \rho \wedge \vartheta;
	\end{split}
\end{equation}
see~\cite{Graham-Lee1988}*{Proposition 1.2}.

\subsection{Fefferman defining functions}
\label{subsection:Fefferman-defining-functions}

Let $\Omega$ be a strictly pseudoconvex domain
in an $(n+1)$-dimensional complex manifold $X$ with $\bdry \Omega = M$.
Assume that $M$ has a pseudo-Einstein contact form $\theta$.
Then there exists a defining function $\rho$ of $\Omega$
such that $d^{c} \rho |_{T M} = \theta$ and
\begin{equation}
	\omega_{+}
	\coloneqq - d d^{c} \log (- \rho)
	= \frac{1 - \kappa \rho}{(- \rho)^{2}} d \rho \wedge \vartheta
		+ \frac{1}{- \rho} \wtxm
\end{equation}
defines a \Kahler metric near the boundary
and satisfies
\begin{equation}
	\Ric(\omega_{+}) + (n + 2) \omega_{+}
	= d d^{c} O(\rho^{n+2});
\end{equation}
see~\cite{Hirachi2014}*{Section 2.2} for example.
Note that this equation is equivalent to
an asymptotic complex Monge-\Ampere equation for the potential.
We call such a $\rho$ a \emph{Fefferman defining function associated with $\theta$}.
Note that $\rho$ is determined uniquely by $\theta$ modulo $O(\rho^{n + 3})$~\cite{Fefferman1976}*{Section II}.
Moreover,
normal derivatives of the curvature and torsion of the Graham-Lee connection
are determined by the boundary data.

\begin{proposition}[\cite{Marugame2021}*{Proposition 3.5}]
\label{prop:normal-derivative-of-GL}
	Let $\wtna$ be the Graham-Lee connection for a Fefferman defining function
	associated with a pseudo-Einstein contact form $\theta$.
	Then the boundary values of
	\begin{equation}
		\wtna_{N}^{p} \wtR_{\alpha \ovxb \rho \ovxs},
		\qquad
		\wtna_{N}^{p} \wtA_{\alpha \beta},
		\qquad
		N^{p - 1} \kappa
		\qquad
		(p \leq n + 1)
	\end{equation}
	are expressed in terms of $R_{\alpha \ovxb \rho \ovxs}$, $A_{\alpha \beta}$,
	and their covariant derivatives with respect to $\theta$.
\end{proposition}

\subsection{Harmonic and pluriharmonic extension}
\label{subsection:harmonic-and-pluriharmonic-extension}

Let $\rho$ be a Fefferman defining function associated with a pseudo-Einstein contact form $\theta$.
Denote by $\Box_{+}$ the $\delb$-Laplacian with respect to the \Kahler form
$\omega_{+} = - d d^{c} \log (- \rho)$;
in our convention,
\begin{equation}
	(\Box_{+} u) \omega_{+}^{n + 1}
	= - (n + 1) d d^{c} u \wedge \omega_{+}^{n}.
\end{equation}
This $\Box_{+}$ can be written in terms of the Graham-Lee connection.

\begin{proposition}[\cite{Graham-Lee1988}*{Proposition 2.1}]
	For any $u \in C^{\infty}(\Omega)$,
	\begin{equation}
	\label{eq:formula-of-Box_+}
		\Box_{+} u
		= - \frac{\rho^{2}}{1 - \kappa \rho}
			\rbra*{N^{2} u + \frac{1}{4} \wtT^{2} u + \kappa N
			+ \frac{1}{2} (\kappa^{\alpha} u_{\alpha}
			+ \kappa^{\ovxb} u_{\ovxb})}
		- \frac{1}{2} \rho \wtxcd_{b} u + n \rho N u,
	\end{equation}
	where
	\begin{equation}
		\wtxcd_{b} u
		\coloneqq - u_{\alpha} {}^{\alpha} - u_{\ovxb} {}^{\ovxb}.
	\end{equation}
\end{proposition}

We use this expression to study normal derivatives of $\Box_{+}$.

\begin{proposition}
\label{prop:normal-derivative-of-Laplacian}
	Let $m$ be a non-negative integer
	and $u$ be a smooth function on $\ovxco$ satisfying $N u = 0$ near $M$.
	For $0 \leq p \leq n + m +1$,
	the boundary value of $N^{p} \Box_{+} (u \rho^{m})$
	is determined by $u|_{M}$, $R_{\alpha \ovxb \rho \ovxs}$,
	$A_{\alpha \beta}$, and their covariant derivatives.
	Moreover,
	\begin{equation}
		[N^{p} \Box_{+} (u \rho^{m})]|_{M}
		=
		\begin{cases}
			0 & (p < m) \\
			m ! m (n + 1 - m) u |_{M} & (p = m).
		\end{cases}
	\end{equation}
\end{proposition}

\begin{proof}
	It follows from \cref{eq:formula-of-Box_+} that
	\begin{align}
		\Box_{+} (u \rho^{m})
		&= - \frac{\rho^{m}}{1 -\kappa \rho} \rbra*{m (m - 1) u + \frac{1}{4} \rho^{2} \wtT^{2} u
			+ m \rho \kappa u + \frac{1}{2} \rho^{2} (\kappa^{\alpha} u_{\alpha}
			+ \kappa^{\ovxb} u_{\ovxb})} \\
		&\quad - \frac{1}{2} \rho^{m + 1} \wtxcd_{b} u + n m \rho^{m} u \\
		&= m (n + 1 - m) \rho^{m} u + O(\rho^{m + 1}).
	\end{align}
	The last equality implies the latter statement.
	The commutators of $\wtna_{N}$ and tangential covariant derivatives are written in term of
	tangential covariant derivatives,
	and the torsion and curvature of the Graham-Lee connection,
	which are expressed by $\wtR_{\alpha \ovxb \rho \ovxs}$,
	$\wtA_{\alpha \beta}$, $\kappa$, and their tangential covariant derivatives.
	This fact and \cref{prop:normal-derivative-of-GL} yield the former statement.
\end{proof}

\begin{theorem}
\label{thm:solution-of-Dirichlet-problem}
	Let $\chi$ be a smooth function on $\ovxco$ with $\chi|_{M} = 0$.
	Then there exist $A, B \in C^{\infty}(\ovxco)$ such that
	$A|_{M} = 0$ and
	\begin{gather}
		\Box_{+} A
		= \chi + O(\rho^{n + 1}), \\
		\Box_{+}(A + B \rho^{n + 1} \log (- \rho))
		= \chi + O(\rho^{n + 2} \log (- \rho)).
	\end{gather}
	Moreover,
	such an $A$ is unique modulo $O(\rho^{n + 1})$ and $B$ is unique modulo $O(\rho)$.
	Furthermore,
	$((N^{p} A)|_{M})_{p = 1}^{n}$ and $B|_{M}$
	are determined by $((N^{l} \chi)|_{M})_{l = 1}^{n + 1}$,
	$R_{\alpha \ovxb \rho \ovxs}$, $A_{\alpha \beta}$,
	and their covariant derivatives.
\end{theorem}

\begin{proof}
	We first construct $u_{1}, \dots , u_{n} \in C^{\infty}(\ovxco)$ inductively
	such that $A_{k} = \sum_{i = 1}^{k} u_{i} \rho^{i}$ satisfies
	\begin{equation}
		\Box_{+} A_{k}
		= \chi + O(\rho^{k + 1}),
	\end{equation}
	$N u_{i} = 0$ near $M$,
	and $u_{i} |_{M}$ is expressed by $((N^{l} \chi)|_{M})_{l = 1}^{i}$,
	$R_{\alpha \ovxb \rho \ovxs}$, $A_{\alpha \beta}$,
	and their covariant derivatives.
	Take $u_{1} \in C^{\infty}(\ovxco)$ such that
	$N u_{1} = 0$ near $M$ and $u_{1} |_{M} = n^{- 1} (N \chi)|_{M}$.
	Then $u_{1}$ satisfies
	\begin{equation}
		\Box_{+} (u_{1} \rho)
		= \chi + O(\rho^{2}).
	\end{equation}
	Assume that we obtain $u_{1}, \dots , u_{k - 1} \in C^{\infty}(\ovxco)$
	such that $A_{k - 1} = \sum_{i = 1}^{k - 1} u_{i} \rho^{i}$ satisfies
	\begin{equation}
		\Box_{+} A_{k - 1}
		= \chi + O(\rho^{k}),
	\end{equation}
	each $u_{i}$ satisfies
	$N u_{i} = 0$ near $M$,
	and $u_{i} |_{M}$ is written in terms of $((N^{l} \chi)|_{M})_{l = 1}^{i}$,
	$R_{\alpha \ovxb \rho \ovxs}$, $A_{\alpha \beta}$,
	and their covariant derivatives.
	If we take $u_{k} \in C^{\infty}(\ovxco)$ so that
	$N u_{k} = 0$ near $M$ and
	\begin{equation}
		u_{k}|_{M}
		= \frac{1}{k ! k (n + 1 - k)}
			\rbra*{(N^{k} \chi)|_{M} - (N^{k} \Box_{+} A_{k - 1})|_{M}},
	\end{equation}
	then $u_{k}|_{M}$ is determined by $((N^{l} \chi)|_{M})_{l = 1}^{k}$,
	$R_{\alpha \ovxb \rho \ovxs}$, $A_{\alpha \beta}$,
	and their covariant derivatives by \cref{prop:normal-derivative-of-Laplacian}.
	Moreover,
	$A_{k} = \sum_{i = 1}^{k} u_{i} \rho^{i}$ satisfies
	\begin{equation}
		\Box_{+} A_{k}
		= \chi + O(\rho^{k + 1}).
	\end{equation}
	Hence $A \coloneqq \sum_{i = 1}^{n} u_{i} \rho^{i}$ satisfies
	\begin{equation}
		\Box_{+} A
		= \chi + O(\rho^{n + 1}).
	\end{equation}
	It follows from the construction that
	$A$ is unique modulo $O(\rho^{n + 1})$.
	Moreover,
	$(N^{p} A)|_{M} = p ! u_{p}|_{M}$ is written in terms of $((N^{l} \chi)|_{M})_{l = 1}^{p}$,
	$R_{\alpha \ovxb \rho \ovxs}$, $A_{\alpha \beta}$,
	and their covariant derivatives.
	Next,
	let $B \in C^{\infty}(\ovxco)$.
	Then we derive from \cref{eq:formula-of-Box_+} that
	\begin{equation}
		\Box_{+} (B \rho^{n + 1} \log (- \rho))
		= - (n + 1) B \rho^{n + 1} + O(\rho^{n + 2} \log (- \rho)).
	\end{equation}
	If we take $B$ so that $N B = 0$ near $M$ and
	\begin{equation}
		B|_{M}
		= - \frac{1}{(n + 1) ! (n + 1)}
			\rbra*{(N^{n + 1} \chi)|_{M} - (N^{n + 1} \Box_{+} A)|_{M}},
	\end{equation}
	then we have
	\begin{equation}
		\Box_{+}(A + B \rho^{n + 1} \log (- \rho))
		= \chi + O(\rho^{n + 2} \log (- \rho)).
	\end{equation}
	This $B$ is unique modulo $O(\rho)$ by the construction
	and \cref{prop:normal-derivative-of-Laplacian}.
	Moreover,
	$B|_{M}$ is determined by $((N^{l} \chi)|_{M})_{l = 1}^{n + 1}$,
	$R_{\alpha \ovxb \rho \ovxs}$, $A_{\alpha \beta}$,
	and their covariant derivatives.
\end{proof}

This result implies the existence of an asymptotic solution to
the Dirichlet problem with respect to $\Box_{+}$.

\begin{proposition}
\label{prop:harmonic-extension}
	Let $f \in C^{\infty}(M)$.
	Then there exist $\wtf, \wtg \in C^{\infty}(\ovxco)$ such that
	$\wtf|_{M} = f$ and
	\begin{gather}
		\Box_{+} \wtf
		= O(\rho^{n + 1}), \\
		\Box_{+} (\wtf + \wtg \rho^{n + 1} \log (- \rho))
		= O(\rho^{n + 2} \log (- \rho)).
	\end{gather}
	Moreover,
	such an $\wtf$ is unique modulo $O(\rho^{n + 1})$
	and $\wtg$ is unique modulo $O(\rho)$.
	Furthermore,
	$((N^{p} \wtf)|_{M})_{p = 0}^{n}$ and $\wtg|_{M}$
	are determined by $f$, $R_{\alpha \ovxb \rho \ovxs}$, $A_{\alpha \beta}$,
	and their covariant derivatives.
\end{proposition}

\begin{definition}[\cite{Hirachi2014}*{Section 3.1}]
	The \emph{critical CR GJMS operator} $P$ is defined by
	$P f \coloneqq - n ! (n + 1) ! \wtg|_{M}$,
	which is expressed by $f$, $R_{\alpha \ovxb \rho \ovxs}$, $A_{\alpha \beta}$,
	and their covariant derivatives.
\end{definition}

\begin{proof}[Proof of \cref{prop:harmonic-extension}]
	Take $\wtf_{0} \in C^{\infty}(\ovxco)$ such that
	$N \wtf_{0} = 0$ near $M$ and $\wtf_{0}|_{M} = f$,
	and set $\chi \coloneqq \Box_{+} \wtf_{0}$.
	Then $\chi|_{M} = 0$ and $(N^{p} \chi)|_{M}$ is expressed by
	$f$, $R_{\alpha \ovxb \rho \ovxs}$, $A_{\alpha \beta}$,
	and their covariant derivatives for any $1 \leq p \leq n + 1$
	by \cref{prop:normal-derivative-of-Laplacian}.
	It follows from \cref{thm:solution-of-Dirichlet-problem} that
	there exist $A, B \in C^{\infty}(\ovxco)$ such that
	$A|_{M} = 0$ and
	\begin{gather}
		\Box_{+} A
		= \chi + O(\rho^{n + 1}), \\
		\Box_{+} (A + B \rho^{n + 1} \log (- \rho))
		= \chi + O(\rho^{n + 2} \log (- \rho)).
	\end{gather}
	If we set $\wtf \coloneqq \wtf_{0} - A$ and $\wtg = - B$,
	then
	\begin{gather}
		\Box_{+} \wtf
		= O(\rho^{n + 1}), \\
		\Box_{+} (\wtf + \wtg \rho^{n + 1} \log (- \rho))
		= O(\rho^{n + 2} \log (- \rho)).
	\end{gather}
	The remaining statements follow from \cref{thm:solution-of-Dirichlet-problem}.
\end{proof}

It is known that any $f \in \scrP$ has a pluriharmonic extension $\wtf$ on the pseudoconvex side;
in particular, $\wtf$ is smooth up to the boundary.
In this case,
$(N^{n + 1} \wtf)|_{M}$ is also determined by the boundary data.

\begin{proposition}
\label{prop:pluriharmonic-extension}
	Let $\wtf$ be the pluriharmonic extension to $\ovxco$ of $f \in \scrP$.
	Then,
	for $0 \leq p \leq n + 1$,
	the boundary value of $N^{p} \wtf$ is determined by $f$,
	$R_{\alpha \ovxb \rho \ovxs}$, $A_{\alpha \beta}$,
	and their covariant derivatives.
\end{proposition}

\begin{proof}
	Take $u_{0}, \dots , u_{n + 1} \in C^{\infty}(\ovxco)$ so that
	$N u_{i} = 0$ near $M$ and
	\begin{equation}
		\wtf
		= \sum_{i = 0}^{n + 1} u_{i} \rho^{i} + O(\rho^{n + 2}).
	\end{equation}
	Since $d d^{c} \wtf = 0$ and $\omega_{+}$ is \Kahler,
	$\Box_{+} \wtf = 0$.
	It follows from \cref{prop:harmonic-extension} that
	$(N^{p} \wtf)|_{M} = p ! u_{p}|_{M}$
	is written in terms of $f$, $R_{\alpha \ovxb \rho \ovxs}$, $A_{\alpha \beta}$,
	and their covariant derivatives for $0 \leq p \leq n$.
	Hence it suffices to show that $(N^{n + 1} \wtf)|_{M}$ is determined by
	$f$, $R_{\alpha \ovxb \rho \ovxs}$, $A_{\alpha \beta}$,
	and their covariant derivatives.
	Since $d d^{c} \wtf = 0$,
	\begin{equation}
		0
		= d d^{c} \wtf (N, \wtT)
		= n (n + 1) u_{n + 1} \rho^{n - 1}
			+ \sum_{i = 0}^{n} d d^{c} (u_{i} \rho^{i}) (N, \wtT)
			+ O(\rho^{n}).
	\end{equation}
	We obtain from \cref{eq:dd^c(NT),prop:normal-derivative-of-GL} that
	\begin{equation}
		(N^{n + 1} \wtf)|_{M}
		= (n + 1)! u_{n + 1}|_{M}
		= - \sum_{i = 0}^{n} \sbra{N^{n - 1} (d d^{c} (u_{i} \rho^{i})(N, \wtT))}|_{M}
	\end{equation}
	is expressed by $f$,
	$R_{\alpha \ovxb \rho \ovxs}$, $A_{\alpha \beta}$,
	and their covariant derivatives.
\end{proof}

\subsection{Burns-Epstein's renormalized connection}
\label{subsection:renormalized-connection}

Fix a local frame of $T^{1, 0} \Omega$.
Let $\psi_{a} {}^{b}$ be the Chern connection form with respect to $\omega_{+}$.
This connection diverges on the boundary since so does $\omega_{+}$.
The \emph{renormalized connection form} $\theta_{a} {}^{b}$ is defined by
\begin{equation}
	\theta_{a} {}^{b}
	= \psi_{a} {}^{b}
		+ \frac{1}{\rho} (\delta_{a} {}^{b} \rho_{c}
		+ \delta_{c} {}^{b} \rho_{a}) \wtxth^{c}. 
\end{equation}
This connection form can be extended smoothly up to the boundary;
see~\cite{Marugame2021}*{Proposition 4.1} for example.
The corresponding curvature form is denoted by $\Theta_{a} {}^{b}$,
which satisfies
\begin{equation}
\label{eq:trace-of-renormalized-curvature}
	\Tr \Theta
	= d d^{c} O(\rho^{n + 2});
\end{equation}
see~\cite{Marugame2016}*{(4.7)}.
Near the boundary,
$\Theta_{a} {}^{b}$ is written in terms of the Graham-Lee connection.

\begin{lemma}[\cite{Marugame2016}*{Proposition 3.5}]
\label{lem:renormalized-connection}
	For the frame $(\wtxth^{\alpha}, \wtxth^{\infty} = \partial \rho)$,
	\begin{gather}
		\theta_{\alpha} {}^{\beta}
		= \wtxo_{\alpha} {}^{\beta}
			+ \frac{1}{2} \kappa (\del \rho - \delb \rho) \delta_{\alpha} {}^{\beta}, \\
		\theta_{\infty} {}^{\beta}
		= \kappa \wtxth^{\beta}
			- \sqrt{- 1} \wtA^{\beta} {}_{\ovxg} \wtxth^{\ovxg}
			- \kappa^{\beta} \delb \rho, \\
		\theta_{\alpha} {}^{\infty}
		= - \wtl_{\alpha \ovxg} \wtxth^{\ovxg}
			- \rho (1 - \kappa \rho)^{- 1} \kappa_{\alpha} \del \rho
			+ \sqrt{- 1} \rho (1 - \kappa \rho)^{- 1} \wtA_{\alpha \beta} \wtxth^{\beta}, \\
		\theta_{\infty} {}^{\infty}
		= - \kappa \delb \rho
			- \rho \kappa^{2} (1 - \kappa \rho)^{- 1} \del \rho
			- \rho (1 - \kappa \rho)^{- 1} \del \kappa.
	\end{gather}
\end{lemma}

Let $\Phi$ be an $\Ad$-invariant polynomial on $\mathfrak{gl}(n + 1, \bbC)$
homogeneous of degree $m$ with $0 \leq m \leq n$.
Although $\Theta_{a} {}^{b}$ is not a $(1, 1)$-form in general,
$\Phi(\Theta)$ is a closed $(m, m)$-form near $M$~\cite{Marugame2021}*{Proposition 4.5(i)}.
If follows from \cref{lem:renormalized-connection} and \cref{eq:TW-curvature-form} that
$\Theta_{a} {}^{b}$ is written in terms of
$\wtR_{\alpha \ovxb \rho \ovxs}$, $ \wtA_{\alpha \beta}$,
$\wtna_{N} \wtA_{\alpha \beta}$, $\kappa$, $N \kappa$,
and their tangential covariant derivatives.
We derive from \cref{prop:normal-derivative-of-GL} the following

\begin{lemma}
\label{lem:normal-derivative-of-renormalized-curvature}
	For $0 \leq p \leq n - 1$,
	the boundary value of $\wtna_{N}^{p} \Phi(\Theta)$ is expressed by
	$R_{\alpha \ovxb \rho \ovxs}$, $A_{\alpha \beta}$,
	and their covariant derivatives.
\end{lemma}

We also consider the change of a pseudo-Einstein contact form $\whxth = e^{\Upsilon} \theta$.
Since both $\theta$ and $\whxth$ are pseudo-Einstein,
$\Upsilon$ must be a CR pluriharmonic function.
Take its pluriharmonic extension $\wtxcu	$.
Then $\whxr = e^{\wtxcu} \rho$ is a Fefferman defining function associated with $\whxth$,
and $\Phi(\Theta)$ is invariant under this change~\cite{Marugame2021}*{Proposition 4.5(i)}.

\begin{remark}
	The function $\rho(z) = \abs{z}^{2} - 1$ on $\bbC^{n + 1}$
	is a Fefferman defining function of the unit ball.
	Moreover,
	the corresponding renormalized connection is flat.
	In particular,
	$\Phi(\Theta) = 0$ near the boundary when $1 \leq \deg \Phi \leq n$.
	The same is true for spherical CR manifolds by the above discussion.
	Hence $P_{\Phi}$, $\calP_{\Phi}$, $P_{\Phi}^{\prime}$, and $Q_{\Phi}^{\prime}$ defined in later sections
	are identically zero on spherical CR manifolds if $1 \leq \deg \Phi \leq n$.
\end{remark}

\section{$P_{\Phi}$-operator and $\calP_{\Phi}$-operator}
\label{section:P_Phi-operator-and-calP_Phi-operator}

In what follows,
let $\Phi$ be an $\Ad$-invariant polynomial on $\mathfrak{gl}(n + 1, \bbC)$
homogeneous of degree $m$ with $0 \leq m \leq n$.
It follows from \cref{prop:harmonic-extension} that,
for $f \in C^{\infty}(M)$,
there exists $\wtf \in C^{\infty}(\ovxco)$ such that
$\wtf|_{M} = f$ and $\Box_{+} \wtf = O(\rho^{n + 1})$.

\begin{proposition}
\label{prop:P-operator}
	Let $f \in C^{\infty}(M)$
	and take $\wtf \in C^{\infty}(\ovxco)$ such that $\wtf|_{M} = f$ and $\Box_{+} \wtf = O(\rho^{n + 1})$.
	Then there exist $A, B \in C^{\infty}(\ovxco)$ such that
	$A|_{M} = 0$ and
	\begin{equation}
		u
		\coloneqq A + B \rho^{n + 1} \log (- \rho)
	\end{equation}
	satisfies
	\begin{equation}
	\label{eq:def-of-P-operator}
		- d d^{c} \wtf \wedge \omega_{+}^{n - m} \wedge \Phi(\Theta) 
		= [\Box_{+} u + O(\rho^{n + 2} \log (- \rho))] \omega_{+}^{n + 1}.
	\end{equation}
	Moreover,
	$B$ is unique modulo $O(\rho)$,
	and $B|_{M}$ is determined only by $f$,
	$R_{\alpha \ovxb \rho \ovxs}$, $A_{\alpha \beta}$,
	and their covariant derivatives.
\end{proposition}

\begin{definition}
\label{def:P_Phi-operator}
	The \emph{$P_{\Phi}$-operator} $P_{\Phi}$ is defined by
	$P_{\Phi} f \coloneqq B|_{M}$.
\end{definition}

\begin{proof}[Proof of \cref{prop:P-operator}]
	We divide into three cases:
	$m = 0$, $m = 1$, and $2 \leq m \leq n$.

	If $m = 0$,
	then $\Phi$ is a constant $c \in \bbC$.
	In this case,
	\begin{equation}
		- (n + 1) d d^{c} \wtf \wedge \omega_{+}^{n} \wedge \Phi(\Theta)
		= \Box_{+} (c \wtf) \omega_{+}^{n + 1}.
	\end{equation}
	Take $\wtg$ in \cref{prop:harmonic-extension}.
	Then $A = 0$ and $B = - c \wtg / (n + 1)$ is a solution of \cref{eq:def-of-P-operator}.
	In particular,
	\begin{equation}
		B|_{M}
		= - \frac{c}{n + 1} \wtg|_{M}
		= \frac{c}{((n + 1)!)^{2}} P f.
	\end{equation}

	If $m = 1$,
	then $\Phi(X_{a} {}^{b}) = c X_{a} {}^{a}$ for some $c \in \bbC$.
	It follows from \cref{eq:trace-of-renormalized-curvature} that
	\begin{equation}
		- d d^{c} \wtf \wedge \omega_{+}^{n - 1} \wedge \Phi(\Theta)
		= O(\rho^{n + 2}) \omega_{+}^{n + 1},
	\end{equation}
	and so $A = B = 0$ is a solution of \cref{eq:def-of-P-operator}.
	\cref{thm:solution-of-Dirichlet-problem} implies the uniqueness of $B$ modulo $O(\rho)$.
	In particular,
	$B|_{M}$ must be zero.

	In the remainder of the proof,
	we assume $2 \leq m \leq n$.
	We first note that
	\begin{gather}
		\omega_{+}^{n + 1}
		= (n + 1) \frac{1 - \kappa \rho}{(- \rho)^{n + 2}} d \rho \wedge \vartheta \wedge \wtxm^{n}, \\
		\omega_{+}^{n - m}
		= (n - m) \frac{1 - \kappa \rho}{(- \rho)^{n - m + 1}} d \rho \wedge \vartheta \wedge \wtxm^{n - m - 1}
			+ \frac{1}{(- \rho)^{n - m}} \wtxm^{n - m}.
	\end{gather}
	This implies that there exists $\chi \in C^{\infty}(\ovxco)$ such that
	\begin{equation}
		- d d^{c} \wtf \wedge \omega_{+}^{n - m} \wedge \Phi(\Theta)
		= \chi \omega_{+}^{n + 1}.
	\end{equation}
	By \cref{thm:solution-of-Dirichlet-problem},
	it suffices to show that $\chi|_{M} = 0$
	and $((N^{p} \chi)|_{M})_{p = 1}^{n + 1}$ are expressed by
	$f$, $R_{\alpha \ovxb \rho \ovxs}$, $A_{\alpha \beta}$,
	and their covariant derivatives.
	We derive from the definition of $\chi$ that
	\begin{align}
		&\chi d \rho \wedge \vartheta \wedge \wtxm^{n} \\
		&= - \frac{n - m}{n + 1} d \rho \wedge \vartheta \wedge \wtxm^{n - m - 1}
			\wedge ((- \rho)^{m + 1} d d^{c} \wtf \wedge \Phi(\Theta)) \\
		&\quad - \frac{1}{n + 1} \wtxm^{n - m}
			\wedge ((- \rho)^{m + 2} (1 - \kappa \rho)^{- 1} d d^{c} \wtf \wedge \Phi(\Theta)).
	\end{align}
	Since the right hand side is equal to zero on $M$,
	we have $\chi|_{M} = 0$.
	We derive from \cref{eq:parallel-differential-forms} that
	\begin{align}
		&(N^{p} \chi)|_{M} (d \rho \wedge \vartheta \wedge \wtxm^{n})|_{M} \\
		&= - \frac{n - m}{n + 1} \sbra{d \rho \wedge \vartheta \wedge \wtxm^{n - m - 1}
			\wedge \wtna_{N}^{p} ((- \rho)^{m + 1} d d^{c} \wtf \wedge \Phi(\Theta))}|_{M} \\
		&\quad - \frac{1}{n + 1} \sbra{\wtxm^{n - m} \wedge \wtna_{N}^{p}((- \rho)^{m + 2} (1 - \kappa \rho)^{- 1}
			d d^{c} \wtf \wedge \Phi(\Theta))}|_{M}.
	\end{align}
	It follows from \cref{prop:harmonic-extension,prop:normal-derivative-of-GL} that
	the right hand side is written in terms of $f$,
	$R_{\alpha \ovxb \rho \ovxs}$, $A_{\alpha \beta}$,
	and their covariant derivatives if $1 \leq p \leq n + 1$.
	Therefore $((N^{p} \chi)|_{M})_{p = 1}^{n + 1}$ are expressed by
	$f$, $R_{\alpha \ovxb \rho \ovxs}$, $A_{\alpha \beta}$,
	and their covariant derivatives.
\end{proof}

Similar to the critical CR GJMS operator,
the $P_{\Phi}$-operator transforms as follows under conformal change.

\begin{proposition}
\label{prop:transformation-law-of-P-operator}
	Let $\whxth = e^{\Upsilon} \theta$ be another pseudo-Einstein contact form.
	Then $e^{(n + 1) \Upsilon} \whP_{\Phi} = P_{\Phi}$,
	where $\whP_{\Phi}$ is defined in terms of $\whxth$.
\end{proposition}

\begin{proof}
	Since both $\theta$ and $\whxth$ are pseudo-Einstein,
	$\Upsilon$ is a CR pluriharmonic function.
	Take its pluriharmonic extension $\wtxcu$.
	Then $\whxr = e^{\wtxcu} \rho$ is a Fefferman defining function associated with $\whxth$.
	Let $\whu = \whA + \whB \whxr^{n + 1} \log (- \whxr)$ be a solution
	of \cref{eq:def-of-P-operator} with respect to $\whxr$.
	Since $\omega_{+}$ and $\Phi(\Theta)$ are invariant under this change,
	\begin{equation}
		\whu
		= \whA + e^{(n + 1) \wtxcu} \whB \wtxcu \rho^{n + 1}
			+ e^{(n + 1) \wtxcu} \whB \rho^{n + 1} \log (- \rho)
	\end{equation}
	satisfies the equation
	\begin{equation}
		- d d^{c} \wtf \wedge \omega_{+}^{n - m + 1} \wedge \Phi(\Theta)
		= \rbra{\Box_{+} \whu + O(\rho^{n + 2} \log (- \rho))} \omega_{+}^{n + 1}.
	\end{equation}
	We derive from \cref{prop:P-operator} that
	$e^{(n + 1) \Upsilon} \whP_{\Phi} f = (e^{(n + 1) \wtxcu} \whB)|_{M} = P_{\Phi} f$.
\end{proof}

Moreover,
the $P_{\Phi}$-operator is symmetric.
Our proof is inspired by that of \cite{Marugame2018}*{Theorem 1.2}.
Before the proof,
we note that
\begin{equation}
	\lp \int_{\rho < - \varepsilon} O(\rho^{n + 2} \log (- \rho)) \omega_{+}^{n + 1}
	= 0
\end{equation}
since the integrand is integrable on the whole $\Omega$.

\begin{proposition}
\label{prop:symmetry-of-P-operatior}
	For any $f_{1}, f_{2} \in C^{\infty}(M)$,
	one has
	\begin{equation}
		\int_{M} f_{1} (P_{\Phi} f_{2}) \, \theta \wedge (d \theta)^{n}
		= \int_{M} f_{2} (P_{\Phi} f_{1}) \, \theta \wedge (d \theta)^{n}.
	\end{equation}
\end{proposition}

\begin{proof}
	Take $\wtf_{i} \in C^{\infty}(\ovxco)$ such that
	$\wtf_{i}|_{M} = f_{i}$ and $\Box_{+} \wtf_{i} = O(\rho^{n + 1})$.
	Then there exist $A_{i}, B_{i} \in C^{\infty}(\ovxco)$ such that
	$A_{i}|_{M} = 0$ and
	\begin{equation}
		u_{i}
		\coloneqq A_{i} + B_{i} \rho^{n + 1} \log (- \rho)
	\end{equation}
	satisfies
	\begin{equation}
		d d^{c} \wtf_{i} \wedge \omega_{+}^{n - m} \wedge \Phi(\Theta)
		= (n + 1) d d^{c} u_{i} \wedge \omega_{+}^{n}
			+ O(\rho^{n + 2} \log (- \rho)) \omega_{+}^{n + 1}.
	\end{equation}
	We have $B_{i}|_{M} = P_{\Phi} f_{i}$ by the definition of $P_{\Phi}$.
	Set
	\begin{align}
		I
		&\coloneqq (n + 1) \lp \int_{\rho < - \varepsilon} d \wtf_{1} \wedge d^{c} u_{2} \wedge \omega_{+}^{n} \\
		&\quad + (n + 1) \lp \int_{\rho < - \varepsilon} d u_{1} \wedge d^{c} \wtf_{2} \wedge \omega_{+}^{n} \\
		&\quad - \lp \int_{\rho < - \varepsilon} d \wtf_{1} \wedge d^{c} \wtf_{2}
			\wedge \omega_{+}^{n - m} \wedge \Phi(\Theta).
	\end{align}
	Since $\Phi(\Theta)$ is an $(m, m)$-form near $M$,
	this $I$ is symmetric in the indices $1$ and $2$.
	We would like to compute this $I$.

	On the one hand,
	\begin{align}
		&(n + 1) \lp \int_{\rho < - \varepsilon} d \wtf_{1} \wedge d^{c} u_{2} \wedge \omega_{+}^{n} \\
		&= (n + 1) \lp \int_{\rho < - \varepsilon}
			d \rbra{\wtf_{1} d^{c} u_{2} \wedge \omega_{+}^{n}}
			- (n + 1) \lp \int_{\rho < - \varepsilon} \wtf_{1} d d^{c} u_{2} \wedge \omega_{+}^{n} \\
		&= (n + 1) \lp \int_{\rho = - \varepsilon}
			\wtf_{1} d^{c} \rbra{A_{2} + B_{2} \rho^{n + 1} \log (- \rho)}
			\wedge (\varepsilon^{- 1} d \vartheta)^{n} \\
		&\quad + \lp \int_{\rho < - \varepsilon} \wtf_{1} (\Box_{+} u_{2}) \omega_{+}^{n + 1} \\
		&= (- 1)^{n} (n + 1)^{2} \int_{M} f_{1} (P_{\Phi} f_{2}) \, \theta \wedge (d \theta)^{n}
			- \lp \int_{\rho < - \varepsilon} \wtf_{1} d d^{c} \wtf_{2}
			\wedge \omega_{+}^{n - m} \wedge \Phi(\Theta),
	\end{align}
	and
	\begin{align}
		&- \lp \int_{\rho < - \varepsilon} d \wtf_{1} \wedge d^{c} \wtf_{2}
			\wedge \omega_{+}^{n - m} \wedge \Phi(\Theta) \\
		&= - \lp \int_{\rho < - \varepsilon} d \sbra{\wtf_{1} d^{c} \wtf_{2}
			\wedge \omega_{+}^{n - m} \wedge \Phi(\Theta)}
			+ \lp \int_{\rho < - \varepsilon} \wtf_{1} d d^{c} \wtf_{2}
			\wedge \omega_{+}^{n - m} \wedge \Phi(\Theta) \\
		&= - \lp \int_{\rho = - \varepsilon} \wtf_{1} d^{c} \wtf_{2}
			\wedge (\varepsilon^{- 1} d \vartheta)^{n - m} \wedge \Phi(\Theta)
			+ \lp \int_{\rho < - \varepsilon} \wtf_{1} d d^{c} \wtf_{2}
			\wedge \omega_{+}^{n - m} \wedge \Phi(\Theta) \\
		&= \lp \int_{\rho < - \varepsilon} \wtf_{1} d d^{c} \wtf_{2}
			\wedge \omega_{+}^{n - m} \wedge \Phi(\Theta).
	\end{align}
	Hence
	\begin{align}
		&(n + 1) \lp \int_{\rho < - \varepsilon} d \wtf_{1} \wedge d^{c} u_{2} \wedge \omega_{+}^{n}
			- \lp \int_{\rho < - \varepsilon} d \wtf_{1} \wedge d^{c} \wtf_{2}
			\wedge \omega_{+}^{n - m} \wedge \Phi(\Theta) \\
		&= (- 1)^{n} (n + 1)^{2} \int_{M} f_{1} (P_{\Phi} f_{2}) \, \theta \wedge (d \theta)^{n}.
	\end{align}
	On the other hand,
	\begin{align}
		&(n + 1) \lp \int_{\rho < - \varepsilon} d u_{1} \wedge d^{c} \wtf_{2} \wedge \omega_{+}^{n} \\
		&= (n + 1) \lp \int_{\rho < - \varepsilon} d \rbra{u_{1} d^{c} \wtf_{2} \wedge \omega_{+}^{n}}
			- (n + 1) \lp \int_{\rho < - \varepsilon} u_{1} d d^{c} \wtf_{2} \wedge \omega_{+}^{n} \\
		&= (n + 1) \lp \int_{\rho = - \varepsilon} (A_{1} + B_{1} (- \varepsilon)^{n + 1} \log \varepsilon)
			d^{c} \wtf_{2} \wedge (\varepsilon^{- 1} d \vartheta)^{n} \\
		& \quad + \lp \int_{\rho < - \varepsilon} u_{1} (\Box_{+} \wtf_{2}) \omega_{+}^{n + 1} \\
		&= \lp \int_{\rho < - \varepsilon} u_{1} (\Box_{+} \wtf_{2}) \omega_{+}^{n + 1}.
	\end{align}
	Since $u_{1}|_{M} = 0$ and $\Box_{+} \wtf_{2} = O(\rho^{n + 1})$,
	the $(n + 1, n + 1)$-form $u_{1} (\Box_{+} \wtf_{2}) \omega_{+}^{n + 1}$ is continuous up to the boundary.
	Thus we have
	\begin{equation}
		(n + 1) \lp \int_{\rho < - \varepsilon} d u_{1} \wedge d^{c} \wtf_{2} \wedge \omega_{+}^{n}
		= 0.
	\end{equation}
	Therefore
	\begin{equation}
		I
		= (- 1)^{n} (n + 1)^{2} \int_{M} f_{1} (P_{\Phi} f_{2}) \, \theta \wedge (d \theta)^{n}.
	\end{equation}
	Since $I$ is symmetric in the indices $1$ and $2$,
	\begin{equation}
		\int_{M} f_{1} (P_{\Phi} f_{2}) \, \theta \wedge (d \theta)^{n}
		= \frac{(- 1)^{n}}{(n + 1)^{2}} I
		= \int_{M} f_{2} (P_{\Phi} f_{1}) \, \theta \wedge (d \theta)^{n},
	\end{equation}
	which completes the proof.
\end{proof}

An argument similar to that in the proof of \cref{prop:P-operator}
gives a symmetric bilinear differential operator $\calP_{\Phi}$.

\begin{proposition}
\label{prop:calP-operator}
	Let $f_{1}, f_{2} \in C^{\infty}(M)$
	and take $\wtf_{i} \in C^{\infty}(\ovxco)$ such that
	$\wtf_{i}|_{M} = f_{i}$ and $\Box_{+} \wtf_{i} = O(\rho^{n + 1})$.
	Then there exist $\calA, \calB \in C^{\infty}(\ovxco)$ such that
	$\calA|_{M} = 0$ and
	\begin{equation}
		u
		\coloneqq \calA + \calB \rho^{n + 1} \log (- \rho)
	\end{equation}
	satisfies
	\begin{equation}
	\label{eq:def-of-calP-operator}
		- d d^{c} (\wtf_{1} \wtf_{2}) \wedge \omega_{+}^{n - m} \wedge \Phi(\Theta) 
		= [\Box_{+} u + O(\rho^{n + 2} \log (- \rho))] \omega_{+}^{n + 1}.
	\end{equation}
	Moreover,
	$\calB$ is unique modulo $O(\rho)$,
	and $\calB|_{M}$ is determined only by $f_{1}$, $f_{2}$,
	$R_{\alpha \ovxb \rho \ovxs}$, $A_{\alpha \beta}$,
	and their covariant derivatives.
\end{proposition}

\begin{definition}
\label{def:calP_Phi-operator}
	The \emph{$\calP_{\Phi}$-operator} $\calP_{\Phi}$ is defined by
	$\calP_{\Phi} (f_{1}, f_{2}) \coloneqq \calB|_{M}$.
\end{definition}

\begin{example}
	If $m = 0$,
	then $\Phi$ is a constant $c \in \bbC$.
	In this case,
	\begin{equation}
		-  (n + 1) d d^{c} (\wtf_{1} \wtf_{2}) \wedge \omega_{+}^{n} \wedge \Phi(\Theta)
		= \Box_{+}(c \wtf_{1} \wtf_{2}) \omega_{+}^{n + 1}.
	\end{equation}
	Hence $\calA$ and $\calB$ satisfy
	\begin{equation}
		\Box_{+} (c \wtf_{1} \wtf_{2} - (n + 1) \calA - (n + 1) \calB \rho^{n + 1} \log (- \rho))
		= O(\rho^{n + 2} \log (- \rho)).
	\end{equation}
	It follows from \cref{prop:harmonic-extension} that
	\begin{equation}
		\calP_{\Phi}(f_{1}, f_{2})
		= \calB|_{M}
		= \frac{c}{((n + 1)!)^{2}} P (f_{1} f_{2})
		= P_{\Phi} (f_{1} f_{2}).
	\end{equation}
\end{example}

An argument similar to that in the proof of \cref{prop:transformation-law-of-P-operator} gives
the transformation rule under conformal change.

\begin{proposition}
\label{prop:transformation-law-of-calP-operator}
	Let $\whxth = e^{\Upsilon} \theta$ be another pseudo-Einstein contact form.
	Then $e^{(n + 1) \Upsilon} \widehat{\calP}_{\Phi} = \calP_{\Phi}$,
	where $\widehat{\calP}_{\Phi}$ is defined in terms of $\whxth$.
\end{proposition}

We also show that the integral of $\calP_{\Phi}(f_{1}, f_{2})$ must vanish.

\begin{proposition}
\label{prop:divergence-of-calP-operator}
	For any $f_{1}, f_{2} \in C^{\infty}(M)$,
	one has
	\begin{equation}
		\int_{M} \calP_{\Phi}(f_{1}, f_{2}) \, \theta \wedge (d \theta)^{n}
		= 0.
	\end{equation}
\end{proposition}

\begin{proof}
	Let $\wtf_{1}$, $\wtf_{2}$, and $u$ be as in \cref{prop:calP-operator}.
	Then
	\begin{align}
		\lp \int_{\rho < - \varepsilon} d d^{c} (\wtf_{1} \wtf_{2})
			\wedge \omega_{+}^{n - m} \wedge \Phi(\Theta)
		&= \lp \int_{\rho < - \varepsilon} d \sbra{d^{c} (\wtf_{1} \wtf_{2})
			\wedge \omega_{+}^{n - m} \wedge \Phi(\Theta)} \\
		&= \lp \int_{\rho = - \varepsilon} d^{c} (\wtf_{1} \wtf_{2})
			\wedge (\varepsilon^{- 1} d \vartheta)^{n - m} \wedge \Phi(\Theta) \\
		&= 0.
	\end{align}
	Hence
	\begin{align}
		0
		&= \lp \int_{\rho < - \varepsilon} (\Box_{+} u) \omega_{+}^{n + 1} \\
		&= - (n + 1) \lp \int_{\rho < - \varepsilon} d d^{c} u \wedge \omega_{+}^{n} \\
		&= - (n + 1) \lp \int_{\rho < - \varepsilon} d \rbra{ d^{c} u \wedge \omega_{+}^{n}} \\
		&= - (n + 1) \lp \int_{\rho = - \varepsilon} d^{c} \rbra{\calA + \calB \rho^{n + 1} \log (- \rho)}
			\wedge (\varepsilon^{- 1} d \vartheta)^{n} \\
		&= (- 1)^{n + 1} (n + 1)^{2} \int_{M} \calP_{\Phi}(f_{1}, f_{2}) \, \theta \wedge (d \theta)^{n},
	\end{align}
	which completes the proof.
\end{proof}

\section{$P_{\Phi}$-prime operator}
\label{section:P_Phi-prime-operator}

Since any $f \in \scrP$ has a pluriharmonic extension,
$P_{\Phi} f = 0$ by the definition of $P_{\Phi}$.
Then we can define the ``secondary'' version of the $P_{\Phi}$-operator,
which is a generalization of the $P$-prime operator introduced in~\cites{Case-Yang2013,Hirachi2014}.

\begin{proposition}
\label{prop:P-prime-operator}
	Let $f \in \scrP$
	and take its pluriharmonic extension $\wtf$.
	Then there exist $A^{\prime}, B^{\prime} \in C^{\infty}(\ovxco)$ such that
	$A^{\prime}|_{M} = 0$ and
	\begin{equation}
		u^{\prime}
		\coloneqq A^{\prime} + B^{\prime} \rho^{n + 1} \log (- \rho)
	\end{equation}
	satisfies
	\begin{equation}
	\label{eq:def-of-P-prime-operator}
		- 2 d \log (- \rho) \wedge d^{c} \wtf \wedge \omega_{+}^{n - m} \wedge \Phi(\Theta) 
		= [\Box_{+} u^{\prime} + O(\rho^{n + 2} \log (- \rho))] \omega_{+}^{n + 1}.
	\end{equation}
	Moreover,
	$B^{\prime}$ is unique modulo $O(\rho)$,
	and $B^{\prime}|_{M}$ is determined only by $f$,
	$R_{\alpha \ovxb \rho \ovxs}$, $A_{\alpha \beta}$,
	and their covariant derivatives.
\end{proposition}

\begin{definition}
\label{def:P_Phi-prime-operator}
	The \emph{$P_{\Phi}$-prime operator} $P_{\Phi}^{\prime}$ is defined by
	$P_{\Phi}^{\prime} f \coloneqq B^{\prime}|_{M}$.
\end{definition}

\begin{proof}[Proof of \cref{prop:P-prime-operator}]
	There exists $\chi \in C^{\infty}(\ovxco)$ such that
	\begin{equation}
		- 2 d \log (- \rho) \wedge d^{c} \wtf \wedge \omega_{+}^{n - m} \wedge \Phi(\Theta)
		= \chi \omega_{+}^{n + 1}.
	\end{equation}
	By \cref{thm:solution-of-Dirichlet-problem},
	it suffices to show that $\chi|_{M} = 0$
	and $((N^{p} \chi)|_{M})_{p = 1}^{n + 1}$ are expressed by
	$f$, $R_{\alpha \ovxb \rho \ovxs}$, $A_{\alpha \beta}$,
	and their covariant derivatives.
	By the definition of $\chi$,
	\begin{equation}
		\chi d \rho \wedge \vartheta \wedge \wtxm^{n}
		= \frac{2}{n + 1} d \rho \wedge \wtxm^{n - m}
			\wedge ((- \rho)^{m + 1} (1 - \kappa \rho)^{- 1} d^{c} \wtf \wedge \Phi(\Theta)).
	\end{equation}
	Since the right hand side is equal to zero on $M$,
	we have $\chi|_{M} = 0$.
	We derive from \cref{eq:parallel-differential-forms} that
	\begin{align}
		&(N^{p} \chi)|_{M} (d \rho \wedge \vartheta \wedge \wtxm^{n})|_{M} \\
		&= \frac{2}{n + 1} [d \rho \wedge \wtxm^{n - m}
			\wedge \wtna_{N}^{p}((- \rho)^{m + 1} (1 - \kappa \rho)^{- 1} d^{c} \wtf \wedge \Phi(\Theta))]|_{M}
	\end{align}
	If $m = 0$,
	then $\Phi$ is a constant.
	It follows from \cref{prop:normal-derivative-of-GL,prop:pluriharmonic-extension} that
	\begin{equation}
		[\wtna_{N}^{p}((- \rho) (1 - \kappa \rho)^{- 1} d^{c} \wtf)]|_{M}
	\end{equation}
	is determined by $f$,
	$R_{\alpha \ovxb \rho \ovxs}$, $A_{\alpha \beta}$,
	and their covariant derivatives if $1 \leq p \leq n + 1$.
	On the other hand,
	if $1 \leq m \leq n$,
	we obtain from \cref{prop:pluriharmonic-extension,prop:normal-derivative-of-GL,lem:normal-derivative-of-renormalized-curvature} that
	\begin{equation}
		[\wtna_{N}^{p}((- \rho)^{m + 1} (1 - \kappa \rho)^{- 1} d^{c} \wtf \wedge \Phi(\Theta))]|_{M}
	\end{equation}
	is written in terms of $f$,
	$R_{\alpha \ovxb \rho \ovxs}$, $A_{\alpha \beta}$,
	and their covariant derivatives if $1 \leq p \leq n + 1$.
	Therefore $((N^{p} \chi)|_{M})_{p = 1}^{n + 1}$ are expressed by
	$f$, $R_{\alpha \ovxb \rho \ovxs}$, $A_{\alpha \beta}$,
	and their covariant derivatives.
\end{proof}

\begin{example}
	If $\deg \Phi = 0$,
	then $\Phi = c \in \bbC$.
	In this case,
	\begin{align}
		&- 2 (n + 1) d \log (- \rho) \wedge d^{c} \wtf \wedge \omega_{+}^{n} \wedge \Phi(\Theta) \\
		&= - (n + 1) d d^{c} (c \wtf \log (- \rho)) \wedge \omega_{+}^{n}
			- (n + 1) c \wtf \omega_{+}^{n + 1} \\
		&= \sbra{\Box_{+} (c \wtf \log (- \rho)) - (n + 1) c \wtf} \omega_{+}^{n + 1}.
	\end{align}
	Hence $A^{\prime}$ and $B^{\prime}$ satisfy
	\begin{equation}
		\Box_{+}(c \wtf \log (- \rho) - (n + 1) A^{\prime} - (n + 1) B^{\prime} \rho^{n + 1} \log (- \rho))
		= (n + 1) c \wtf + O(\rho^{n + 2} \log (- \rho)).
	\end{equation}
	\cite{Hirachi2014}*{Lemma 4.4} implies that
	\begin{equation}
		P_{\Phi}^{\prime} f
		= \frac{c}{((n + 1)!)^{2}} P^{\prime} f,
	\end{equation}
	where $P^{\prime}$ is the $P$-prime operator defined in
	\cites{Case-Yang2013,Hirachi2014}.
\end{example}

The transformation rule of $P_{\Phi}^{\prime}$ under conformal change
is written in terms of $\calP_{\Phi}$.

\begin{proposition}
\label{prop:transformation-law-of-P-prime-operator}
	Let $\whxth = e^{\Upsilon} \theta$ be another pseudo-Einstein contact form.
	Then
	\begin{equation}
		e^{(n + 1) \Upsilon} \whP_{\Phi}^{\prime} f
		= P_{\Phi}^{\prime} f + \calP_{\Phi}(\Upsilon, f).
	\end{equation}
\end{proposition}

\begin{proof}
	Since both $\theta$ and $\whxth$ are pseudo-Einstein,
	$\Upsilon$ is a CR pluriharmonic function.
	Take its pluriharmonic extension $\wtxcu$.
	Then $\whxr = e^{\wtxcu} \rho$ is a Fefferman defining function associated with $\whxth$.
	Let $u^{\prime} = A^{\prime} + B^{\prime} \rho^{n + 1} \log (- \rho)$
	and $\whu^{\prime} = \whA^{\prime} + \whB^{\prime} \whxr^{n + 1} \log (- \whxr)$ be solutions
	of \cref{eq:def-of-P-prime-operator} with respect to $\rho$ and $\whxr$ respectively.
	Since $\omega_{+}$ and $\Phi(\Theta)$ are invariant under this change,
	\begin{align}
		& (\Box_{+} \whu^{\prime} + O(\whxr^{n + 2} \log (- \whxr))) \omega_{+}^{n + 1} \\
		&= - 2 d \log (- \whxr) \wedge d^{c} \wtf \wedge \omega_{+}^{n - m} \wedge \Phi(\Theta) \\
		&= - 2 d \log (- \rho) \wedge d^{c} \wtf \wedge \omega_{+}^{n - m} \wedge \Phi(\Theta)
			- 2 d \wtxcu \wedge d^{c} \wtf \wedge \omega_{+}^{n - m} \wedge \Phi(\Theta) \\
		&= \rbra{\Box_{+}u^{\prime} + O(\rho^{n + 2} \log (- \rho))} \omega_{+}^{n + 1}
			- d d^{c} (\wtxcu  \wtf) \wedge \omega_{+}^{n - m} \wedge \Phi(\Theta);
	\end{align}
	in the last equality,
	we use the fact that $\wtf$ and $\wtxcu$ are pluriharmonic.
	This implies that
	\begin{equation}
		u
		\coloneqq \whu^{\prime} - u^{\prime}
		= \rbra{\whA^{\prime} - A^{\prime} + e^{(n + 1) \wtxcu} \whB^{\prime} \wtxcu \rho^{n + 1}}
			+ \rbra{e^{(n + 1) \wtxcu} \whB^{\prime} - B^{\prime}} \rho^{n + 1} \log (- \rho)
	\end{equation}
	satisfies the equation
	\begin{equation}
		- d d^{c} (\wtxcu  \wtf) \wedge \omega_{+}^{n - m} \wedge \Phi(\Theta)
		= \sbra{\Box_{+} u + O(\rho^{n + 2} \log (- \rho))} \omega_{+}^{n + 1}.
	\end{equation}
	If follows from \cref{prop:calP-operator} that
	\begin{equation}
		e^{(n + 1) \Upsilon} \whP_{\Phi}^{\prime} f - P_{\Phi}^{\prime} f
		= (e^{(n + 1) \wtxcu} \whB^{\prime} - B^{\prime})|_{M}
		= \calP_{\Phi}(\Upsilon, f),
	\end{equation}
	which completes the proof.
\end{proof}

We next show that the integral of $P_{\Phi}^{\prime} f$ must vanish.

\begin{proposition}
\label{prop:divergence-of-P-prime-operator}
	For any $f \in \scrP$,
	one has
	\begin{equation}
		\int_{M} (P_{\Phi}^{\prime} f) \, \theta \wedge (d \theta)^{n}
		= 0.
	\end{equation}
\end{proposition}

\begin{proof}
	Let $\wtf$ be the pluriharmonic extension of $f$.
	Then
	\begin{align}
		\lp \int_{\rho < - \varepsilon} d \log (- \rho)
			\wedge d^{c} \wtf \wedge \omega_{+}^{n - m} \wedge \Phi(\Theta)
		&= \lp \int_{\rho < - \varepsilon} d \sbra{\log (- \rho)
			d^{c} \wtf \wedge \omega_{+}^{n - m} \wedge \Phi(\Theta)} \\
		&= \lp \log \varepsilon \cdot \int_{\rho = - \varepsilon} d^{c} \wtf
			\wedge \omega_{+}^{n - m} \wedge \Phi(\Theta).
	\end{align}
	When $m < n$,
	the $2 (n - m)$-form $\omega_{+}^{n - m}$ is $d$-exact on $\Set{\rho = - \varepsilon}$.
	On the other hand,
	it follows from \cite{Takeuchi2020-Chern}*{Theorem 1.1} that
	$[\Phi(\Theta)] = 0$ in $H^{2 m} (M; \bbC)$ if $2 m \geq n + 2$.
	In particular,
	$\Phi(\Theta)$ is $d$-exact on $\Set{\rho = - \varepsilon}$ in the case of $m = n \geq 2$.
	In addition,
	the equality \eqref{eq:trace-of-renormalized-curvature} gives that
	$\Phi(\Theta)$ is $d$-exact when $m = n = 1$.
	Hence $d^{c} \wtf \wedge \omega_{+}^{n - m} \wedge \Phi(\Theta)$ is $d$-exact 
	on $\Set{\rho = - \varepsilon}$.
	The Stokes theorem implies that
	\begin{equation}
		\lp \int_{\rho < - \varepsilon} d \log (- \rho) \wedge d^{c} \wtf
			\wedge \omega_{+}^{n - m} \wedge \Phi(\Theta)
		= 0.
	\end{equation}
	Therefore
	\begin{align}
		0
		&= \lp \int_{\rho < - \varepsilon} (\Box_{+} u^{\prime}) \omega_{+}^{n + 1} \\
		&= - (n + 1) \lp \int_{\rho < - \varepsilon} d d^{c} u^{\prime} \wedge \omega_{+}^{n} \\
		&= - (n + 1) \lp \int_{\rho = - \varepsilon}
			d^{c} \rbra{A^{\prime} + B^{\prime} \rho^{n + 1} \log (- \rho)}
			\wedge (\varepsilon^{- 1} d \vartheta)^{n} \\
		&= (- 1)^{n + 1} (n + 1)^{2} \int_{M} (P_{\Phi}^{\prime} f) \, \theta \wedge (d \theta)^{n},
	\end{align}
	which completes the proof.
\end{proof}

Marugame~\cite{Marugame2018}*{Theorem 1.2} has proved that
the $P$-prime operator is symmetric.
We generalize this result to the $P_{\Phi}$-prime operator when $m < n$;
see \cref{section:deg-Phi=n-case} for the $m = n$ case.

\begin{theorem}
\label{thm:symmetric-of-P-prime}
	If $m < n$,
	then
	\begin{equation}
		\int_{M} f_{1} (P_{\Phi}^{\prime} f_{2}) \, \theta \wedge (d \theta)^{n}
		= \int_{M} f_{2} (P_{\Phi}^{\prime} f_{1}) \, \theta \wedge (d \theta)^{n}
	\end{equation}
	for any $f_{1}, f_{2} \in \scrP$.
\end{theorem}

\begin{proof}
	Let $\wtf_{i}$ be a pluriharmonic extension of $f_{i}$.
	Take $A_{i}^{\prime}, B_{i}^{\prime} \in C^{\infty}(\ovxco)$ such that
	$A_{i}^{\prime}|_{M} = 0$ and
	\begin{equation}
		u_{i}^{\prime}
		\coloneqq A_{i}^{\prime} + B_{i}^{\prime} \rho^{n + 1} \log (- \rho)
	\end{equation}
	satisfies
	\begin{equation}
		- 2 d \log (- \rho) \wedge d^{c} \wtf_{i} \wedge \omega_{+}^{n - m} \wedge \Phi(\Theta) 
		= [\Box_{+} u_{i}^{\prime} + O(\rho^{n + 2} \log (- \rho))] \omega_{+}^{n + 1}.
	\end{equation}
	For $\varepsilon > 0$,
	set
	\begin{equation}
		I_{\varepsilon}
		\coloneqq \int_{\rho < - \varepsilon} \rbra{d \wtf_{1} \wedge d^{c} u_{2}^{\prime}
			+ d u_{1}^{\prime} \wedge d^{c} \wtf_{2}} \wedge \omega_{+}^{n}.
	\end{equation}
	This is symmetric in the indices $1$ and $2$.
	Consider the logarithmic term of $I_{\varepsilon}$ as $\varepsilon \to + 0$.

	On the one hand,
	\begin{align}
		\lp \int_{\rho < - \varepsilon} d u_{1}^{\prime} \wedge d^{c} \wtf_{2} \wedge \omega_{+}^{n}
		&= \lp \int_{\rho < - \varepsilon} d \rbra{u_{1}^{\prime} d^{c} \wtf_{2} \wedge \omega_{+}^{n}} \\
		&= \lp \int_{\rho = - \varepsilon}
			\rbra{A_{1}^{\prime} + B_{1}^{\prime} (- \varepsilon)^{n + 1} \log \varepsilon}
			d^{c} \wtf_{2} \wedge (\varepsilon^{- 1} d \vartheta)^{n} \\
		&= 0.
	\end{align}

	On the other hand,
	\begin{align}
		&\int_{\rho < - \varepsilon} d \wtf_{1} \wedge d^{c} u_{2}^{\prime} \wedge \omega_{+}^{n} \\
		&= \int_{\rho < - \varepsilon} d \rbra{\wtf_{1} d^{c} u_{2}^{\prime} \wedge \omega_{+}^{n}}
			- \int_{\rho < - \varepsilon} \wtf_{1} d d^{c} u_{2}^{\prime} \wedge \omega_{+}^{n} \\
		&= \int_{\rho = - \varepsilon} \wtf_{1} d^{c} (A_{2}^{\prime} + B_{2}^{\prime} \rho^{n + 1} \log (- \rho))
			\wedge (\varepsilon^{- 1} d \vartheta)^{n} \\
		&\quad - \frac{2}{n + 1} \int_{\rho < - \varepsilon}
			\wtf_{1} d \log (- \rho) \wedge d^{c} \wtf_{2} \wedge \omega_{+}^{n - m} \wedge \Phi(\Theta) \\
		&\quad + \int_{\rho < - \varepsilon} O(\rho^{n + 2} \log (- \rho)) \omega_{+} ^{n + 1}.
	\end{align}
	The logarithmic part of the first term is
	\begin{equation}
	\label{eq:symmetric-in-indices}
		(- 1)^{n} (n + 1) \int_{M} f_{1} (P_{\Phi}^{\prime} f_{2}) \, \theta \wedge (d \theta)^{n},
	\end{equation}
	and that of the third term is equal to zero.
	We consider the second term.
	Extend $\Phi(\Theta)$ to a smooth $(m, m)$-form on $\Omega$ for computation.
	In what follows,
	$(\text{cpt supp})$ stands for a compactly supported form on $\Omega$.
	\begin{align}
		&\int_{\rho < - \varepsilon}
			\wtf_{1} d \log (- \rho) \wedge d^{c} \wtf_{2} \wedge \omega_{+}^{n - m} \wedge \Phi(\Theta) \\
		&= \int_{\rho < - \varepsilon}
			d \sbra{\wtf_{1} \log (- \rho) d^{c} \wtf_{2} \wedge \omega_{+}^{n - m} \wedge \Phi(\Theta)} \\
		&\quad - \int_{\rho < - \varepsilon}
			\log (- \rho) d \wtf_{1} \wedge d^{c} \wtf_{2} \wedge \omega_{+}^{n - m} \wedge \Phi(\Theta)
			+ \int_{\rho < - \varepsilon} (\text{cpt supp}).
	\end{align}
	The second term is symmetric in the indices $1$ and $2$
	while the third term contains no $\log \varepsilon$ term.
	We can compute the first term as follows:
	\begin{align}
		&\int_{\rho < - \varepsilon}
			d \sbra{\wtf_{1} \log (- \rho) d^{c} \wtf_{2} \wedge \omega_{+}^{n - m} \wedge \Phi(\Theta)} \\
		&= \varepsilon^{m - n} \log \varepsilon \int_{\rho = - \varepsilon}
			\wtf_{1} d^{c} \wtf_{2} \wedge (d \vartheta)^{n - m} \wedge \Phi(\Theta) \\
		&= \varepsilon^{m - n} \log \varepsilon \int_{\rho < - \varepsilon}
			d \wtf_{1} \wedge d^{c} \wtf_{2} \wedge (d \vartheta)^{n - m} \wedge \Phi(\Theta) \\
		&\quad + \varepsilon^{m - n} \log \varepsilon
			\int_{\rho < - \varepsilon} (\text{cpt supp}).
	\end{align}
	The first term is symmetric in the indices $1$ and $2$
	while the second term contains no $\log \varepsilon$ term;
	here we use the assumption $m < n$.
	Therefore \cref{eq:symmetric-in-indices} should be symmetric in the indices $1$ and $2$.
\end{proof}

\section{$Q_{\Phi}$-prime curvature}
\label{section:Q_Phi-prime-curvature}

In this section,
we introduce the $Q_{\Phi}$-prime curvature,
a generalization of the $Q$-prime curvature defined in \cites{Case-Yang2013,Hirachi2014}.

\begin{proposition}
\label{prop:Q-prime-curvature}
	There exist $F^{\prime}, G^{\prime} \in C^{\infty}(\ovxco)$ such that
	$F^{\prime}|_{M} = 0$ and
	\begin{equation}
		v^{\prime}
		\coloneqq F^{\prime} + G^{\prime} \rho^{n + 1} \log (- \rho)
	\end{equation}
	satisfies
	\begin{equation}
	\label{eq:def-of-Q-prime-curvature}
		\frac{2}{n - m + 1} \rbra*{\frac{d d^{c} \rho}{- \rho}}^{n - m + 1} \wedge \Phi(\Theta)
		= [\Box_{+} v^{\prime} + O(\rho^{n + 2} \log (- \rho))] \omega_{+}^{n + 1}.
	\end{equation}
	Moreover,
	$G^{\prime}$ is unique modulo $O(\rho)$,
	and $G^{\prime}|_{M}$ is determined only by
	$R_{\alpha \ovxb \rho \ovxs}$, $A_{\alpha \beta}$,
	and their covariant derivatives.
\end{proposition}

\begin{definition}
\label{def:Q_Phi-prime-curvature}
	The \emph{$Q_{\Phi}$-prime curvature} $Q_{\Phi}^{\prime}$ is defined by
	$Q_{\Phi}^{\prime} \coloneqq G^{\prime}|_{M}$.
\end{definition}

\begin{proof}[Proof of \cref{prop:Q-prime-curvature}]
	There exists $\chi \in C^{\infty}(\ovxco)$ such that
	\begin{equation}
		\frac{2}{n - m + 1} \rbra*{\frac{d d^{c} \rho}{- \rho}}^{n - m + 1} \wedge \Phi(\Theta)
		= \chi \omega_{+}^{n + 1}.
	\end{equation}
	By \cref{thm:solution-of-Dirichlet-problem},
	it suffices to show that $\chi|_{M} = 0$
	and $((N^{p} \chi)|_{M})_{p = 1}^{n + 1}$ are expressed by
	$R_{\alpha \ovxb \rho \ovxs}$, $A_{\alpha \beta}$,
	and their covariant derivatives.
	By the definition of $\chi$,
	\begin{align}
		\chi d \rho \wedge \vartheta \wedge \wtxm^{n}
		&= \frac{2}{n + 1} d \rho \wedge \vartheta \wedge \wtxm^{n - m}
			\wedge ((- \rho)^{m + 1} \kappa (1 - \kappa \rho)^{- 1} \Phi(\Theta)) \\
		&\quad + \frac{2}{(n + 1) (n - m + 1)} \wtxm^{n - m + 1}
			\wedge ((- \rho)^{m + 1} (1 - \kappa \rho)^{- 1} \Phi(\Theta)).
	\end{align}
	Since the right hand side is equal to zero on $M$,
	we have $\chi|_{M} = 0$.
	We derive from \cref{eq:parallel-differential-forms} that
	\begin{align}
		&(N^{p} \chi)|_{M} (d \rho \wedge \vartheta \wedge \wtxm^{n})|_{M} \\
		&= \frac{2}{n + 1} [d \rho \wedge \vartheta \wedge \wtxm^{n - m}
			\wedge \wtna_{N}^{p} ((- \rho)^{m + 1} \kappa (1 - \kappa \rho)^{- 1} \Phi(\Theta))]|_{M} \\
		&\quad + \frac{2}{(n + 1) (n - m + 1)} [\wtxm^{n - m + 1}
			\wedge \wtna_{N}^{p} ((- \rho)^{m + 1} (1 - \kappa \rho)^{- 1} \Phi(\Theta))]|_{M}.
	\end{align}
	If $m = 0$,
	then $\Phi$ is a constant
	and $\wtxm^{n + 1} = 0$.
	It follows from \cref{prop:normal-derivative-of-GL} that
	\begin{equation}
		[\wtna_{N}^{p}((- \rho) \kappa (1 - \kappa \rho)^{- 1})]|_{M}
	\end{equation}
	is expressed by $R_{\alpha \ovxb \rho \ovxs}$, $A_{\alpha \beta}$,
	and their covariant derivatives if $1 \leq p \leq n + 1$.
	On the other hand,
	if $1 \leq m \leq n$,
	we obtain from \cref{prop:normal-derivative-of-GL,lem:normal-derivative-of-renormalized-curvature} that
	\begin{equation}
		[\wtna_{N}^{p}((- \rho)^{m + 1} \kappa (1 - \kappa \rho)^{- 1} \Phi(\Theta))]|_{M},
		\qquad
		[\wtna_{N}^{p} ((- \rho)^{m + 1} (1 - \kappa \rho)^{- 1} \Phi(\Theta))]|_{M}
	\end{equation}
	are written in terms of $R_{\alpha \ovxb \rho \ovxs}$, $A_{\alpha \beta}$,
	and their covariant derivatives if $1 \leq p \leq n + 1$.
	Therefore $((N^{p} \chi)|_{M})_{p = 1}^{n + 1}$ are expressed by
	$R_{\alpha \ovxb \rho \ovxs}$, $A_{\alpha \beta}$,
	and their covariant derivatives.
\end{proof}

\begin{example}
	If $\Phi = c \in \bbC$,
	\begin{align}
		2 \rbra*{\frac{d d^{c} \rho}{- \rho}}^{n + 1} \wedge \Phi(\Theta)
		&= - 2 c (n + 1) d \log (- \rho) \wedge d^{c} \log (- \rho) \wedge \omega_{+}^{n}
			+ 2 c \omega_{+}^{n + 1} \\
		&= c (2 - \abs*{d \log (- \rho)}^{2}) \omega_{+}^{n + 1}.
	\end{align}
	Thus we have
	\begin{equation}
		\Box_{+}((n + 1) F^{\prime} + (n + 1) G^{\prime} \rho^{n + 1} \log (- \rho))
		= c \rbra{2 - \abs*{d \log (- \rho)}^{2}} + O(\rho^{n + 2} \log (- \rho)).
	\end{equation}
	It follows from \cite{Hirachi2014}*{(5.5)} that
	\begin{equation}
		Q_{\Phi}^{\prime}
		= \frac{c}{((n + 1)!)^{2}} Q^{\prime},
	\end{equation}
	where $Q^{\prime}$ is the $Q$-prime curvature defined in
	\cites{Case-Yang2013,Hirachi2014}.
\end{example}

The $Q_{\Phi}$-prime curvature has an analogous transformation law to the $Q$-prime curvature.

\begin{proposition}
\label{prop:transformation-law-of-Q-prime-curvature}
	Let $\whxth = e^{\Upsilon} \theta$ be another pseudo-Einstein contact form.
	Then
	\begin{equation}
		e^{(n + 1) \Upsilon} \whQ_{\Phi}^{\prime} 
		= Q_{\Phi}^{\prime} + 2 P_{\Phi}^{\prime} \Upsilon + \calP_{\Phi}(\Upsilon, \Upsilon).
	\end{equation}
\end{proposition}

\begin{proof}
	We first note that
	\begin{equation}
		\rbra*{\frac{d d^{c} \rho}{- \rho}}^{n - m + 1}
		= - (n - m + 1) d \log (- \rho) \wedge d^{c} \log (- \rho) \wedge \omega_{+}^{n - m}
			+ \omega_{+}^{n - m + 1}.
	\end{equation}
	Since both $\theta$ and $\whxth$ are pseudo-Einstein,
	$\Upsilon$ is a CR pluriharmonic function.
	Take its pluriharmonic extension $\wtxcu$.
	Then $\whxr = e^{\wtxcu} \rho$ is a Fefferman defining function associated with $\whxth$.
	Let $v^{\prime} = F^{\prime} + G^{\prime} \rho^{n + 1} \log (- \rho)$
	and $\whv^{\prime} = \whF^{\prime} + \whG^{\prime} \whxr^{n + 1} \log (- \whxr)$ be a solution
	of \cref{eq:def-of-Q-prime-curvature} with respect to $\rho$ and $\whxr$ respectively.
	Take also a solution $u^{\prime} = A^{\prime} + B^{\prime} \rho^{n + 1} \log (- \rho)$
	of \cref{eq:def-of-P-prime-operator} with respect to $\wtxcu$.
	Since $\omega_{+}$ and $\Phi(\Theta)$ are invariant under the change of Fefferman defining functions,
	\begin{align}
		& \sbra*{\Box_{+} \whv^{\prime} + O(\whxr^{n + 2} \log (- \whxr))} \omega_{+}^{n + 1} \\
		&= - 2 d \log (- \whxr) \wedge d^{c} \log (- \whxr) \wedge \omega_{+}^{n - m} \wedge \Phi(\Theta)
			+ \frac{2}{n - m + 1} \omega_{+}^{n - m + 1} \wedge \Phi(\Theta) \\
		&= - 2 d \log (- \rho) \wedge d^{c} \log (- \rho)
			\wedge \omega_{+}^{n - m} \wedge \Phi(\Theta) 
			+ \frac{2}{n - m + 1} \omega_{+}^{n - m + 1} \wedge \Phi(\Theta) \\
		&\quad - 4 d \log (- \rho) \wedge d^{c} \wtxcu
			\wedge \omega_{+}^{n - m} \wedge \Phi(\Theta)
			- 2 d \wtxcu \wedge d^{c} \wtxcu \wedge \omega_{+}^{n - m} \wedge \Phi(\Theta) \\
		&= \sbra{\Box_{+} (v^{\prime} + 2 u^{\prime}) + O(\rho^{n + 2} \log (- \rho))} \omega_{+}^{n + 1} 
			- d d^{c} (\wtxcu^{2}) \wedge \omega_{+}^{n - m} \wedge \Phi(\Theta).
	\end{align}
	This implies that
	\begin{align}
		u
		&\coloneqq \whv^{\prime} - v^{\prime} - 2 u^{\prime} \\
		&= \rbra{\whF^{\prime} - F^{\prime} - 2 A^{\prime}
			+ e^{(n + 1) \wtxcu} \whG^{\prime} \wtxcu \rho^{n + 1}}
			+ \rbra{e^{(n + 1) \wtxcu} \whG^{\prime} - G^{\prime} - 2 B^{\prime}} \rho^{n + 1} \log (- \rho)
	\end{align}
	satisfies the equation
	\begin{equation}
		- d d^{c} (\wtxcu^{2}) \wedge \omega_{+}^{n - m} \wedge \Phi(\Theta)
		= \sbra{\Box_{+} u + O(\rho^{n + 2} \log (- \rho))} \omega_{+}^{n + 1}.
	\end{equation}
	If follows from \cref{prop:calP-operator} that
	\begin{equation}
		e^{(n + 1) \Upsilon} \whQ_{\Phi}^{\prime} - Q_{\Phi}^{\prime} - 2 P_{\Phi}^{\prime} \Upsilon
		= (e^{(n + 1) \wtxcu} \whG^{\prime} - G^{\prime} - 2 B^{\prime})|_{M}
		= \calP_{\Phi}(\Upsilon, \Upsilon),
	\end{equation}
	which completes the proof.
\end{proof}

\begin{proof}[Proof of \cref{thm:total-Q_Phi-prime-curvature}]
	The former statement follows from
	\cref{prop:divergence-of-calP-operator,prop:divergence-of-P-prime-operator,prop:transformation-law-of-Q-prime-curvature}.
	We show the latter statement.
	To this end,
	take $F^{\prime}$ and $G^{\prime}$ as in \cref{prop:Q-prime-curvature}.

	On the one hand,
	\begin{align}
		&\lp \int_{\rho < - \varepsilon} [\Box_{+} (F^{\prime} + G^{\prime} \rho^{n + 1} \log (- \rho))]
			\omega_{+}^{n + 1} \\
		&= - (n + 1) \lp \int_{\rho < - \varepsilon} d d^{c} (F^{\prime} + G^{\prime} \rho^{n + 1} \log (- \rho))
			\wedge \omega_{+}^{n} \\
		&= - (n + 1) \lp \int_{\rho < - \varepsilon}
			d \sbra{d^{c} (F^{\prime} + G^{\prime} \rho^{n + 1} \log (- \rho)) \wedge \omega_{+}^{n}} \\
		&= - (n + 1) \lp \int_{\rho = - \varepsilon} d^{c} (F^{\prime} + G^{\prime} \rho^{n + 1} \log (- \rho))
			\wedge (\varepsilon^{- 1} d \vartheta)^{n} \\
		&= (- 1)^{n + 1} (n + 1)^{2} \ovQ_{\Phi}^{\prime}.
	\end{align}

	On the other hand,
	\begin{align}
		&\lp \int_{\rho < - \varepsilon} \frac{2}{n - m + 1}
			\rbra*{\frac{d d^{c} \rho}{- \rho}}^{n - m + 1} \wedge \Phi(\Theta) \\
		&= - 2 \lp \int_{\rho < - \varepsilon}
			d \log (- \rho) \wedge d^{c} \log (- \rho) \wedge \omega_{+}^{n - m} \wedge \Phi(\Theta) \\
		&\quad + \frac{2}{n - m + 1} \lp \int_{\rho < - \varepsilon} \omega_{+}^{n - m + 1} \wedge \Phi(\Theta).
	\end{align}
	Here
	\begin{align}
		\lp \int_{\rho < - \varepsilon} \omega_{+}^{n - m + 1} \wedge \Phi(\Theta)
		&= \lp \int_{\rho < - \varepsilon} d \sbra*{\frac{\vartheta}{- \rho} \wedge \omega_{+}^{n - m}
			\wedge \Phi(\Theta)} \\
		&= \lp \varepsilon^{- n + m - 1} \int_{\rho = - \varepsilon}
			\vartheta \wedge (d \vartheta)^{n - m} \wedge \Phi(\Theta) \\
		&= 0.
	\end{align}
	Therefore we have the desired equality.
\end{proof}

\section{$\deg \Phi = n$ case}
\label{section:deg-Phi=n-case}

In this section,
we consider the case of $\deg \Phi = n$.
Similar to \cref{prop:P-operator},
the $\calP_{\Phi}$-operator is identically zero.
We would like to compare $P_{\Phi}^{\prime}$ and $Q_{\Phi}^{\prime}$
with $X^{\Phi}_{\alpha}$ and $\calI_{\Phi}^{\prime}$ introduced by Marugame~\cite{Marugame2021}
and Case and the author~\cite{Case-Takeuchi2023}.

We first recall the definitions of $X^{\Phi}_{\alpha}$ and $\calI_{\Phi}^{\prime}$.
It follows from \cite{Marugame2021} that
\begin{gather}
	\Phi(\Theta)|_{T M}
	= S^{\Phi} (d \theta)^{n}
		+ n^{2} (S^{\Phi}_{\alpha \ovin} \theta^{\alpha}
		+ S^{\Phi}_{\infty \ovxb} \theta^{\ovxb}) \wedge \theta \wedge (d \theta)^{n - 1}, \\
	d \theta \wedge (N \contr \Phi(\Theta))|_{T M}
	= n S^{\Phi}_{\infty \ovin} \theta \wedge (d \theta)^{n},
\end{gather}
where $S^{\Phi}$, $S^{\Phi}_{\alpha \ovin}$, $S^{\Phi}_{\infty \ovxb}$, and $S^{\Phi}_{\infty \ovin}$
are written in terms of $S_{\alpha \ovxb \rho \ovxs}$ and its covariant derivatives.
Note that our sign of $N$ is different from Marugame's one.
The $(1, 0)$-form $X^{\Phi}_{\alpha}$ and the $\calI_{\Phi}$-prime curvature $\calI_{\Phi}^{\prime}$ is given by
\begin{equation}
	X^{\Phi}_{\alpha}
	\coloneqq S^{\Phi}_{\alpha \ovin} - \frac{1}{n^{2}} \nabla_{\alpha} S^{\Phi},
	\qquad
	\calI_{\Phi}^{\prime}
	\coloneqq S^{\Phi}_{\infty \ovin} + \frac{\Scal}{n^{2} (n + 1)} S^{\Phi}
		+ \frac{1}{n^{3}} \Delta_{b} S^{\Phi}.
\end{equation}

We next consider a relation between the $P_{\Phi}$-prime operator and $X^{\Phi}_{\alpha}$.
Let $f \in \scrP$
and $\wtf$ be its pluriharmonic extension.
It follows from \cite{Takeuchi2022}*{Lemma 3.1} that
\begin{equation}
	(d^{c} \wtf)|_{T M}
	= \frac{\sqrt{-1}}{2} \rbra*{f_{\ovxb} \theta^{\ovxb}
		- f_{\alpha} \theta^{\alpha}} + \frac{1}{2 n} (\Delta_{b} f) \theta.
\end{equation}
Hence
\begin{equation}
	(d^{c} \wtf \wedge \Phi(\Theta))|_{T M}
	= - \frac{n}{2} \rbra*{S^{\Phi}_{\alpha \ovin} f^{\alpha}
		+ S^{\Phi}_{\infty \ovxb} f^{\ovxb}
		- \frac{1}{n^{2}} S^{\Phi} \Delta_{b} f} (d^{c} \rho \wedge (d d^{c} \rho)^{n})|_{T M}.
\end{equation}
Thus we have
\begin{align}
	&- 2 d \log (- \rho) \wedge d^{c} \wtf \wedge \Phi(\Theta) \\
	&= (- 1)^{n} \frac{n}{n + 1}
		\rbra*{S^{\Phi}_{\alpha \ovin} f^{\alpha}
		+ S^{\Phi}_{\infty \ovxb} f^{\ovxb}
		- \frac{1}{n^{2}} S^{\Phi} \Delta_{b} f} \rho^{n + 1} \omega_{+}^{n + 1} \\
	&\quad + O(\rho^{n + 2}) \omega_{+}^{n + 1}.
\end{align}
\cref{prop:P-prime-operator} and the proof of \cref{thm:solution-of-Dirichlet-problem} yield that
\begin{align}
	P_{\Phi}^{\prime} f
	&= (- 1)^{n + 1} \frac{n}{(n + 1)^{2}}
		\rbra*{S^{\Phi}_{\alpha \ovin} f^{\alpha}
		+ S^{\Phi}_{\infty \ovxb} f^{\ovxb}
		- \frac{1}{n^{2}} S^{\Phi} \Delta_{b} f} \\
	&= (- 1)^{n + 1} \frac{n}{(n + 1)^{2}}
		\sbra*{X^{\Phi}_{\alpha} f^{\alpha}
		+ X^{\Phi}_{\ovxb} f^{\ovxb}
		+ \frac{1}{n^{2}} \rbra*{(S^{\Phi} f_{\alpha})_{,} {}^{\alpha}
		+ (S^{\Phi} f_{\ovxb})_{,} {}^{\ovxb}}}.
	\label{eq:P-prime-and-X}
\end{align}
This means that $P_{\Phi}^{\prime}$ coincides with
a constant multiple of
$X^{\Phi}_{\alpha} \nabla^{\alpha} + X^{\Phi}_{\ovxb} \nabla^{\ovxb}$ modulo a divergence term.
Moreover,
the proof of \cref{thm:symmetric-of-P-prime} implies that
$P_{\Phi}^{\prime}$ is symmetric if and only if
\begin{align}
	&\int_{M} \rbra*{\wtf_{1} d^{c} \wtf_{2} \wedge \Phi(\Theta)}|_{T M} \\
	&= - \frac{n}{2} \int_{M} f_{1} \rbra*{S^{\Phi}_{\alpha \ovin} (f_{2})^{\alpha}
		+ S^{\Phi}_{\infty \ovxb} (f_{2})^{\ovxb}
		- \frac{1}{n^{2}} S^{\Phi} \Delta_{b} f_{2}} \, \theta \wedge (d \theta)^{n}
\end{align}
is symmetric in the indices $1$ and $2$.
\cref{eq:P-prime-and-X} yields that
\begin{align}
	&\int_{M} f_{1} \rbra{S^{\Phi}_{\alpha \ovin} (f_{2})^{\alpha}
		+ S^{\Phi}_{\infty \ovxb} (f_{2})^{\ovxb}
		- \frac{1}{n^{2}} S^{\Phi} \Delta_{b} f_{2}}\, \theta \wedge (d \theta)^{n} \\
	&= \int_{M} f_{1} \rbra{X^{\Phi}_{\alpha} (f_{2})^{\alpha} + X^{\Phi}_{\ovxb} (f_{2})^{\ovxb}}
		\, \theta \wedge (d \theta)^{n} \\
	&\quad - \frac{1}{n^{2}} \int_{M} S^{\Phi}
		\sbra{(f_{1})_{\alpha} (f_{2})^{\alpha} + (f_{1})_{\ovxb} (f_{2})^{\ovxb}}
		\, \theta \wedge (d \theta)^{n}.
\end{align}
Hence $P_{\Phi}^{\prime}$ is symmetric
if and only if so is $X^{\Phi}_{\alpha} \nabla^{\alpha} + X^{\Phi}_{\ovxb} \nabla^{\ovxb}$,
which has been discussed in \cites{Case-Gover2020,Marugame2021,Case-Takeuchi2023}.

We finally discuss the $Q_{\Phi}$-prime curvature and the $\calI_{\Phi}$-prime curvature.
It follows from \cite{Marugame2021}*{Proposition 6.5 and Proof of Theorem 6.6} that
\begin{align}
	[N \contr (d d^{c} \rho \wedge \Phi(\Theta))]|_{T M}
	&= \kappa|_{M} \theta \wedge \Phi(\Theta)|_{T M}
		+ d \theta \wedge (N \contr \Phi(\Theta))|_{T M} \\
	&= n \rbra*{S^{\Phi}_{\infty \ovin}
		+ \frac{\Scal}{n^{2} (n + 1)} S^{\Phi}} (d^{c} \rho \wedge (d d^{c} \rho)^{n})|_{T M}.
\end{align}
Hence
\begin{equation}
	2 \rbra*{\frac{d d^{c} \rho}{- \rho}} \wedge \Phi(\Theta)
	= (- 1)^{n + 1} \frac{2 n}{n + 1}
		\rbra*{S^{\Phi}_{\infty \ovin} + \frac{\Scal}{n^{2} (n + 1)} S^{\Phi}}
		\rho^{n + 1} \omega_{+}^{n + 1}
		+ O(\rho^{n + 2}) \omega_{+}^{n + 1}.
\end{equation}
\cref{prop:Q-prime-curvature} and the proof of \cref{thm:solution-of-Dirichlet-problem} imply that
\begin{align}
	Q_{\Phi}^{\prime}
	&= (- 1)^{n} \frac{2 n}{(n + 1)^{2}}
		\rbra*{S^{\Phi}_{\infty \ovin} + \frac{\Scal}{n^{2} (n + 1)} S^{\Phi}} \\
	&= (- 1)^{n} \frac{2 n}{(n + 1)^{2}} \calI_{\Phi}^{\prime}
	 + (- 1)^{n + 1} \frac{2}{n^{2} (n + 1)^{2}} \Delta_{b} S^{\Phi}.
\end{align}
In particular,
$Q_{\Phi}^{\prime}$ is equal to a constant multiple $\calI_{\Phi}^{\prime}$
modulo a divergence term.

\section{$\deg \Phi = n - 1$ case on Sasakian $\eta$-Einstein manifolds}
\label{section:deg-Phi=n-1-case-on-SE}

In this section,
we consider the case of $\deg \Phi = n - 1$ on Sasakian $\eta$-Einstein manifolds.

Let $(S, T^{1, 0} S, \eta)$ be a $(2 n + 1)$-dimensional Sasakian $\eta$-Einstein manifold
with Einstein constant $(n + 1) \lambda$.
Then
\begin{equation}
\label{eq:Fefferman-defining-function-of-Sasakian-eta-Einstein}
	\rho
	=
	\begin{cases}
		\lambda^{- 1} (r^{2 \lambda} - 1) & \lambda \neq 0, \\
		\log r^{2} & \lambda = 0,
	\end{cases}
\end{equation}
is a Fefferman defining function of $\Set{r < 1}$ in $C(S)$ associated with $\eta$%
~\cite{Takeuchi2018}*{Proposition 3.1}.
Note that
\begin{equation}
	d \rho
	= (1 + \lambda \rho) d \log r^{2},
	\qquad
	\vartheta
	= d^{c} \rho
	= (1 + \lambda \rho) \eta.
\end{equation}
Let $(\eta, \theta^{\alpha}, \theta^{\ovxb})$ be an admissible coframe on $S$.
Then $\theta^{\alpha}$ (resp.\ $\theta^{\ovxb}$) defines a $(1, 0)$-form (resp.\ $(0, 1)$-form) on $C(S)$,
and
\begin{equation}
	d d^{c} \rho
	= \sqrt{- 1}(1 + \lambda \rho) l_{\alpha \ovxb} \theta^{\alpha} \wedge \theta^{\ovxb}
		+ \lambda (1 + \lambda \rho)^{- 1} d \rho \wedge \vartheta.
\end{equation}
In particular,
\begin{equation}
\label{eq:Levi-form-and-transverse-curvature-for-Sasakian-eta-Einstein}
	\wtl_{\alpha \ovxb}
	= (1 + \lambda \rho) l_{\alpha \ovxb},
	\qquad
	\kappa
	= \lambda (1 + \lambda \rho)^{- 1}.
\end{equation}
We compute the Graham-Lee connection with respect to $\rho$.
\cref{eq:str-eq-of-TW-conn1,eq:str-eq-of-TW-conn2} yield that
\begin{align}
	d \theta^{\beta}
	&= \theta^{\alpha} \wedge \omega_{\alpha} {}^{\beta}
	= \theta^{\alpha} \wedge \rbra*{\omega_{\alpha} {}^{\beta}
		+ \frac{1}{2} \lambda (1 + \lambda \rho)^{- 1} d \rho \cdot \delta_{\alpha} {}^{\beta}}
		+ \frac{1}{2} \lambda (1 + \lambda \rho)^{- 1} d \rho \wedge \theta^{\beta}, \\
	d \wtl_{\alpha \ovxb}
	&= \lambda d \rho \cdot l_{\alpha \ovxb}
		+ (1 + \lambda \rho) d l_{\alpha \ovxb} \\
	&= \rbra*{\omega_{\alpha} {}^{\gamma}
		+ \frac{1}{2} \lambda (1 + \lambda \rho)^{- 1}
			d \rho \cdot \delta_{\alpha} {}^{\gamma}} \wtl_{\gamma \ovxb}
		+ \wtl_{\alpha \ovxg} \rbra*{\omega_{\ovxb} {}^{\ovxg}
		+ \frac{1}{2} \lambda (1 + \lambda \rho)^{- 1} d \rho \cdot \delta_{\ovxb} {}^{\ovxg}}.
\end{align}
Hence the uniqueness of the Graham-Lee connection implies
\begin{equation}
\label{eq:GL-connection-and-torsion-for-Sasakian-eta-Einstein}
	\wtxo_{\alpha} {}^{\beta}
	= \omega_{\alpha} {}^{\beta} + \frac{1}{2} \lambda
		(1 + \lambda \rho)^{- 1} d \rho \cdot \delta_{\alpha} {}^{\beta},
	\qquad
	\wtA_{\alpha \beta}
	= 0.
\end{equation}
In particular,
\begin{equation}
\label{eq:normal-derivative-of-(1,0)-tensor}
	\wtna_{N} f_{\alpha}
	= (N f)_{\alpha} - \frac{1}{2} \lambda (1 + \lambda \rho)^{- 1} f_{\alpha}
\end{equation}
for any $f \in C^{\infty}(C(S))$.
The curvature form $\wtxco_{\alpha} {}^{\beta}$ of $\wtna$ is given by
\begin{equation}
	\wtxco_{\alpha} {}^{\beta}
	= \Omega_{\alpha} {}^{\beta}
	= R_{\alpha} {}^{\beta} {}_{\rho \ovxs} \theta^{\rho} \wedge \theta^{\ovxs}.
\end{equation}
Hence
\begin{equation}
	\wtR_{\alpha} {}^{\beta} {}_{\rho \ovxs}
	= R_{\alpha} {}^{\beta} {}_{\rho \ovxs},
	\qquad
	\wtna_{N} \wtR_{\alpha} {}^{\beta} {}_{\rho \ovxs}
	= - \lambda (1 + \lambda \rho)^{- 1} \wtR_{\alpha} {}^{\beta} {}_{\rho \ovxs}.
\end{equation}
Let $\wtS_{\alpha} {}^{\beta} {}_{\rho \ovxs}$ be
the completely trace-free part of $\wtR_{\alpha} {}^{\beta} {}_{\rho \ovxs}$.
This satisfies
\begin{equation}
\label{eq:normal-derivative-of-Chern-tensor}
	\wtS_{\alpha} {}^{\beta} {}_{\rho \ovxs}
	= S_{\alpha} {}^{\beta} {}_{\rho \ovxs},
	\qquad
	\wtna_{N} \wtS_{\alpha} {}^{\beta} {}_{\rho \ovxs}
	= - \lambda (1 + \lambda \rho)^{- 1} \wtS_{\alpha} {}^{\beta} {}_{\rho \ovxs}.
\end{equation}

Consider the renormalized connection with respect to $\rho$.

\begin{lemma}[\cite{Takeuchi2020-Formulae-preprint}*{Lemma 5.1}]
\label{lem:renormalized-connection-for-Sasakian-eta-Einstein}
	For a Fefferman defining function $\rho$
	given by \cref{eq:Fefferman-defining-function-of-Sasakian-eta-Einstein},
	the components of the renormalized curvature satisfy
	\begin{equation}
		\Theta_{\alpha} {}^{\beta}
		= \wtS_{\alpha} {}^{\beta} {}_{\rho \ovxs} \theta^{\rho} \wedge \theta^{\ovxs},
		\qquad
		\Theta_{\infty} {}^{\beta}
		= 0,
		\qquad
		\Theta_{\alpha} {}^{\infty}
		= 0,
		\qquad
		\Theta_{\infty} {}^{\infty}
		= 0.
	\end{equation}
\end{lemma}

We obtain a smooth function $\wtcS^{\Phi}$ on $C(S)$ satisfying
\begin{equation}
	\wtxm \wedge \Phi(\Theta)
	= \wtcS^{\Phi} \wtxm^{n},
\end{equation}
which is written in terms of $\wtS_{\alpha \ovxb \rho \ovxs}$ and $\Phi$.
We write $\calS^{\Phi}$ for the boundary value of $\wtcS^{\Phi}$,
which is written in terms of $S_{\alpha \ovxb \rho \ovxs}$ and $\Phi$.
We derive from $\deg \Phi = n - 1$ and \cref{eq:normal-derivative-of-Chern-tensor} that
\begin{equation}
\label{eq:normal-derivative-of-Phi-curvature}
	N \wtcS^{\Phi}
	= - (n - 1) \lambda (1 + \lambda \rho)^{- 1} \wtcS^{\Phi}.
\end{equation}

The \Kahler form $\omega_{+}$ is given by
\begin{equation}
	\omega_{+}
	= - d d^{c} \log (- \rho)
	= \frac{1}{\rho^{2} (1 + \lambda \rho)} d \rho \wedge d^{c} \rho
		+ \frac{1}{- \rho} \wtxm.
\end{equation}
Note that
\begin{equation}
	\omega_{+}^{n + 1}
	= (- 1)^{n} \frac{n + 1}{\rho^{n + 2} (1 + \lambda \rho)} d \rho \wedge d^{c} \rho \wedge \wtxm^{n}.
\end{equation}
Consider the $\delb$-Laplacian $\Box_{+}$ with respect to $\omega_{+}$.
It follows from \cref{eq:formula-of-Box_+} that
\begin{equation}
\label{eq:formula-of-Box-on-SE}
	\Box_{+} u
	= - \rho^{2} (1 + \lambda \rho) \rbra*{N^{2} u + \frac{1}{4} \wtT^{2} u
		+ \frac{\lambda}{1 + \lambda \rho} N u}
		- \frac{\rho}{2} \wtxcd_{b} u + n \rho N u.
\end{equation}
This implies that
\begin{equation}
\label{eq:Box-n-th}
	\Box_{+} (A \rho^{n})
	= n A \rho^{n} + \rbra*{- n^{2} \lambda A - \frac{1}{2} \wtxcd_{b} A} \rho^{n + 1} + O(\rho^{n + 2}),
\end{equation}
and
\begin{equation}
\label{eq:Box-log}
	\Box_{+} (B \rho^{n + 1} \log (- \rho))
	= - (n + 1) B \rho^{n + 1} + O(\rho^{n + 2} \log (- \rho))
\end{equation}
for $A, B \in C^{\infty}(C(S))$ with $N A = N B = 0$ near $S$.


Now we consider the $P_{\Phi}$-prime operator.

\begin{proposition}
	The $P_{\Phi}$-prime operator $P_{\Phi}^{\prime}$ on $(S, T^{1, 0} S, \eta)$ is given by
	\begin{equation}
		P_{\Phi}^{\prime} f
		= \frac{(- 1)^{n - 1}}{2 n^{2} (n + 1)^{2}}
			\sbra*{(\Delta_{b}^{2} f) \calS^{\Phi}
			+ 2 n \lambda (\Delta_{b} f) \calS^{\Phi}
			+ \Delta_{b} ((\Delta_{b} f) \calS^{\Phi})}.
	\end{equation}
\end{proposition}

\begin{proof}
	Let $f \in \scrP$ and take its pluriharmonic extension $\wtf$.
	Then
	\begin{align}
		- 2 d \log (- \rho) \wedge d^{c} \wtf \wedge \omega_{+} \wedge \Phi(\Theta)
		&= - 2 d \wtf \wedge d^{c} \log (- \rho) \wedge \omega_{+} \wedge \Phi(\Theta) \\
		&= \frac{2}{\rho^{2}} (N \wtf) d \rho \wedge d^{c} \rho \wedge \wtxm \wedge \Phi(\Theta) \\
		&= \frac{2}{\rho^{2}} (N \wtf) \wtcS^{\Phi} d \rho \wedge d^{c} \rho \wedge \wtxm^{n} \\
		&= (- 1)^{n} \frac{2 (1 + \lambda \rho)}{n + 1} (N \wtf) \wtcS^{\Phi} \rho^{n} \omega_{+}^{n + 1}.
	\end{align}
	Set
	\begin{equation}
		\varphi
		\coloneqq (- 1)^{n} \frac{2 (1 + \lambda \rho)}{n + 1} (N \wtf) \wtcS^{\Phi}.
	\end{equation}
	It follows from $d d^{c} \wtf = 0$ and \cref{eq:formula-of-Box-on-SE} that
	$(N \wtf)|_{S} = (2 n)^{- 1} \Delta_{b} f$.
	Moreover,
	\cref{eq:dd^c(NT)} and $(\Delta_{b}^{2} + n^{2} T^{2}) f = 0$ (see~\cite{Graham-Lee1988}*{Section 3})
	imply that
	\begin{equation}
		(N^{2} \wtf)|_{S}
		= - \frac{1}{4} (\wtT^{2} \wtf)|_{S} - \lambda (N \wtf)|_{S}
		= \frac{1}{4 n^{2}} \Delta_{b}^{2} f - \frac{\lambda}{2 n} \Delta_{b} f.
	\end{equation}
	Thus we have
	\begin{equation}
		\varphi|_{S}
		= (- 1)^{n} \frac{2}{n + 1} (N \wtf)|_{S} \wtcS^{\Phi}|_{S}
		= (- 1)^{n} \frac{1}{n (n + 1)} (\Delta_{b} f) \calS^{\Phi},
	\end{equation}
	and
	\begin{align}
		(N \varphi)|_{S}
		&= (- 1)^{n} \frac{2 \lambda}{n + 1} (N \wtf)|_{S} \wtcS^{\Phi}|_{S}
			+ (- 1)^{n} \frac{2}{n + 1} (N^{2} \wtf)|_{S} \wtcS^{\Phi}|_{S} \\
		&\quad + (- 1)^{n} \frac{2}{n + 1} (N \wtf)|_{S} (N \wtcS^{\Phi})|_{S} \\
		&= (- 1)^{n} \frac{\lambda}{n (n + 1)} (\Delta_{b} f) \calS^{\Phi}
			+ (- 1)^{n} \frac{1}{2 n^{2} (n + 1)} \rbra*{\Delta_{b}^{2} f - 2 n \lambda \Delta_{b} f}
			\calS^{\Phi} \\
		&\quad + (- 1)^{n + 1} \frac{(n - 1) \lambda}{n (n + 1)} (\Delta_{b} f) \calS^{\Phi} \\
		&= (- 1)^{n} \frac{1}{2 n^{2} (n + 1)} (\Delta_{b}^{2} f) \calS^{\Phi}
			+ (- 1)^{n + 1} \frac{(n - 1) \lambda}{n (n + 1)} (\Delta_{b} f) \calS^{\Phi}.
	\end{align}
	For the computation of $P_{\Phi}^{\prime} f$,
	it suffices to find $A^{\prime}, B^{\prime} \in C^{\infty}(C(S))$ such that
	$N A = N B = 0$ near $S$ and
	\begin{equation}
		\Box_{+} \rbra*{A^{\prime} \rho^{n} + B^{\prime} \rho^{n + 1} \log(- \rho)}
		= \varphi \rho^{n} + O(\rho^{n + 2} \log(- \rho)).
	\end{equation}
	This implies the simultaneous equations
	\begin{equation}
		n A^{\prime}|_{S}
		= \varphi|_{S},
		\qquad
		- n^{2} \lambda A^{\prime}|_{S} - \frac{1}{2} \Delta_{b} A^{\prime}|_{S} - (n + 1) B^{\prime}|_{S}
		= (N \varphi)|_{S}.
	\end{equation}
	It follows from the definition of $P_{\Phi}^{\prime} f$ that
	\begin{equation}
		P_{\Phi}^{\prime} f
		= B^{\prime}|_{S}
		= \frac{(- 1)^{n - 1}}{2 n^{2} (n + 1)^{2}}
			\sbra*{(\Delta_{b}^{2} f) \calS^{\Phi}
			+ 2 n \lambda (\Delta_{b} f) \calS^{\Phi}
			+ \Delta_{b} ((\Delta_{b} f) \calS^{\Phi})},
	\end{equation}
	which completes the proof.
\end{proof}

We next consider the $Q_{\Phi}$-prime curvature.

\begin{proposition}
	The $Q_{\Phi}$-prime curvature $Q_{\Phi}^{\prime}$ on $(S, T^{1, 0} S, \eta)$ is given by
	\begin{equation}
		Q_{\Phi}^{\prime}
		= \frac{(- 1)^{n - 1}}{n (n + 1)^{2}}
			\rbra*{2 n \lambda^{2} \calS^{\Phi} + \lambda \Delta_{b} \calS^{\Phi}}.
	\end{equation}
\end{proposition}

\begin{proof}
	\begin{align}
		\rbra*{\frac{d d^{c} \rho}{- \rho}}^{2} \wedge \Phi(\Theta)
		&= \frac{2 \lambda}{\rho^{2} (1 + \lambda \rho)}
			d \rho \wedge d^{c} \rho \wedge \wtxm \wedge \Phi(\Theta) \\
		&= (- 1)^{n} \frac{2 \lambda}{n + 1} \wtcS^{\Phi} \rho^{n} \omega_{+}^{n + 1}.
	\end{align}
	Set
	\begin{equation}
		\psi
		\coloneqq (- 1)^{n} \frac{2 \lambda}{n + 1} \wtcS^{\Phi}.
	\end{equation}
	By \cref{eq:normal-derivative-of-Phi-curvature},
	we have
	\begin{equation}
		\psi|_{S}
		= (- 1)^{n} \frac{2 \lambda}{n + 1} \calS^{\Phi},
		\qquad
		(N \psi)|_{S}
		= (- 1)^{n - 1} \frac{2 (n - 1) \lambda^{2}}{n + 1} \calS^{\Phi}.
	\end{equation}
	For the computation of $Q_{\Phi}^{\prime}$,
	it suffices to find $F^{\prime}, G^{\prime} \in C^{\infty}(C(S))$ such that
	\begin{equation}
		\Box_{+} \rbra*{F^{\prime} \rho^{n} + G^{\prime} \rho^{n + 1} \log(- \rho)}
		= \psi \rho^{n} + O(\rho^{n + 2} \log(- \rho)).
	\end{equation}
	This implies the simultaneous equations
	\begin{equation}
		n F^{\prime}|_{S}
		= \psi|_{S},
		\qquad
		- n^{2} \lambda F^{\prime}|_{S} - \frac{1}{2} \Delta_{b} F^{\prime}|_{S} - (n + 1) G^{\prime}|_{S}
		= (N \psi)|_{S}.
	\end{equation}
	It follows from the definition of $Q_{\Phi}^{\prime}$ that
	\begin{equation}
		Q_{\Phi}^{\prime}
		= G^{\prime}|_{S}
		= \frac{(- 1)^{n - 1}}{n (n + 1)^{2}}
			\rbra*{2 n \lambda^{2} \calS^{\Phi} + \lambda \Delta_{b} \calS^{\Phi}},
	\end{equation}
	which completes the proof.
\end{proof}

\section*{Acknowledgements}

The author would like to thank Kengo Hirachi and Taiji Marugame for helpful comments.
He would also like to thank the referee for a careful reading and some valuable suggestions.

\bibliography{my-reference,my-reference-preprint}

\end{document}